\documentclass[a4paper,12pt]{amsart}

\usepackage{graphicx}
\usepackage{tikz-cd}
\usepackage{amssymb}
\usepackage{mathtools}
\usepackage{amsthm}
\usepackage[normalem]{ulem}
\usetikzlibrary{arrows}
\usepackage{enumitem}
\usepackage{xfrac}
\usepackage{hyperref, aliascnt}
\usepackage[capitalise, nameinlink, noabbrev, nosort]{cleveref}

\usepackage[
    margin=1.25in
]{geometry}

\hypersetup{pdfauthor={Kristin Courtney}, pdftitle={Completely Positive Approximations and Inductive Systems}, colorlinks=true, linkcolor=darkgray, linkbordercolor=darkgray, citecolor=darkgray, citebordercolor=darkgray, linktocpage=true}

\newtheorem{theorem}{Theorem}[section]

\newaliascnt{lemma}{theorem}
\newtheorem{lemma}[lemma]{Lemma}
\aliascntresetthe{lemma}

\newaliascnt{proposition}{theorem}
\newtheorem{proposition}[proposition]{Proposition}
\aliascntresetthe{proposition}

\newaliascnt{corollary}{theorem}
\newtheorem{corollary}[corollary]{Corollary}
\aliascntresetthe{corollary}

\newaliascnt{claim}{theorem}

\aliascntresetthe{claim}

\newtheorem*{theorem*}{Theorem}

\theoremstyle{definition}

\newaliascnt{definition}{theorem}
\newtheorem{definition}[definition]{Definition}
\aliascntresetthe{definition}

\newtheorem*{definition*}{Definition}

\newaliascnt{notation}{theorem}

\aliascntresetthe{notation}

\newaliascnt{remark}{theorem}
\newtheorem{remark}[remark]{Remark}
\aliascntresetthe{remark}

\newaliascnt{example}{theorem}
\newtheorem{example}[example]{Example}
\aliascntresetthe{example}

\newaliascnt{remarks}{theorem}
\newtheorem{remarks}[remarks]{Remarks}
\aliascntresetthe{remarks}

\numberwithin{equation}{section}

\newcommand{\sbt}{\,\begin{picture}(-1,1)(-1,-2.5)\circle*{3}\end{picture}\,\,\, }

\allowdisplaybreaks

\begin{document}
\def\C{\mathbb{C}}
\def\M{\mathrm{M}}
\def\Cstar{\mathrm{C}^*}
\def\CPCstar{\mathrm{CPC}^*}
\def\Lim{\overset{\tikz \draw [arrows = {->[harpoon]}] (-1,0) -- (0,0);}{(F_n,\rho_n)}_n}
\def\Lima{\overset{\tikz \draw [arrows = {->[harpoon]}] (-1,0) -- (0,0);}{(A_n,\rho_n)}_n}
\def\Limar{\overset{\tikz \draw [arrows = {->[harpoon]}] (-2.25,0) -- (0,0);}{(\M_r(A_n),\rho^{(r)}_n)}_n}
\def\Limaj{\overset{\tikz \draw [arrows = {->[harpoon]}] (-1.25,0) -- (0,0);}{(A_{n_j},\hat{\rho}_{n_j})}_j}
\def\Limj{\overset{\tikz \draw [arrows = {->[harpoon]}] (-1.25,0) -- (0,0);}{(F_{n_j},\hat{\rho}_{n_j})}_j}

\title{Completely positive approximations and inductive systems}
 \author[K.\ Courtney]{Kristin Courtney}
     \address{Mathematical Institute, University of M\"unster, Einsteinstr.\ 62, 48149 M\"unster, Germany}
     \email{kcourtne@uni-muenster.de}
 \thanks{
This research was supported by the Deutsche Forschungsgemeinschaft (DFG, German Research Foundation) under Germany's Excellence Strategy -- EXC 2044 -- 390685587, Mathematics M\"unster -- Dynamics -- Geometry -- Structure, the Deutsche Forschungsgemeinschaft (DFG, German Research Foundation) -- Project-ID 427320536 -- SFB 1442, and ERC Advanced Grant 834267 -- AMAREC}
\date{\today}

\begin{abstract}
We consider inductive systems of $\Cstar$-algebras with completely positive contractive connecting maps. We define a condition, called $\Cstar$-encoding, which is sufficient for the limit of the system to be completely order isomorphic to a $\Cstar$-algebra and hence guarantees a unique $\Cstar$-algebra associated to the limit. When the system consists of finite-dimensional $\Cstar$-algebras, this condition is also necessary and thus characterizes when the limit is completely order isomorphic to a (nuclear) $\Cstar$-algebra. $\Cstar$-encoding systems generalize the NF systems of Blackadar and Kirchberg and the $\CPCstar$-systems of the author and Winter. Moreover, any system of completely positive approximations of a nuclear $\Cstar$-algebra gives rise to a $\Cstar$-encoding system. Consequently a separable $\Cstar$-algebra is nuclear if and only if it is completely order isomorphic to the limit of a $\Cstar$-encoding system. This gives an inductive limit description of all separable nuclear $\Cstar$-algebras equivalent to the recent construction of the author and Winter but without the additional structure of order zero maps. Without these extra structural requirements, one can easily construct examples of our systems, which we demonstrate for all amenable group $\Cstar$-algebras. 
\end{abstract}

\maketitle

\section*{Introduction}
\renewcommand*{\thetheorem}{\Alph{theorem}}
\renewcommand*{\thecorollary}{\Alph{corollary}}
\renewcommand*{\thedefinition}{\Alph{definition}}

Inductive limit constructions are ubiquitous in operator algebras because they enable the transfer of structural information from relatively well-behaved building blocks to more general algebras. This transfer is particularly well utilized in the  classification of von Neumann and $\Cstar$-algebras (see \cite{Ell, EG96, Li20, Lin03, MvN} to name a few). To do so, one must first discern whether a given algebra is isomorphic to one arising from an inductive limit. 
The definitive result in this direction for von Neumann algebras was established in \cite{Con76}, 
where Connes showed that all injective, all amenable, and all semi-discrete von Neumann algebras contain an ultraweak$^*$-dense nested sequence of finite-dimensional von Neumann algebras, i.e., they are all approximately finite-dimensional (called AFD or hyperfinite). The direct analogue in the $\Cstar$-setting fails:  
Any $\Cstar$-algebra which arises as the inductive limit of finite-dimensional $\Cstar$-algebras (called an AF algebra) is nuclear, where nuclearity is considered the $\Cstar$-analogue to amenability and semi-discreteness, but most nuclear $\Cstar$-algebras are not AF.

Nonetheless, generalizations of inductive systems of finite-dimensional $\Cstar$- algebras have proven to be a powerful tool in the study of $\Cstar$-algebras. One way to generalize is to weaken the finite-dimensional assumption on the algebras, 
giving rise to systems of (sub)homogeneous $\Cstar$-algebras. 
Another route is to instead weaken the assumptions on the connecting maps of the system, giving rise to the NF systems of Blackadar and Kirchberg in \cite{BK97} and more recently the $\CPCstar$-systems introduced by the author and Winter in \cite{CW1}. 
It is this route we pursue here.

Classically, an inductive system consists of a sequence of $\Cstar$-algebras $(A_n)_n$ together with connecting $^*$-homomorphisms $\rho_{n+1,n}\colon A_n\to A_{n+1}$. These induce $^*$-homomorphisms $\rho_n\colon A_n\to \prod_m A_m/\bigoplus_m A_m$ into the $\Cstar$-algebra of norm bounded sequences modulo null convergent sequences, and the inductive limit of the system is simply the norm closure of these images $\overline{\bigcup_n\rho_n(A_n)}$. To generalize this construction, 
we relax the assumptions on the connecting maps from $^*$-homomorphisms to completely positive and contractive (c.p.c.)\ maps.  
Though not necessarily multiplicative, c.p.c.\ maps are $^*$-linear and positivity preserving (even up to matrix amplifications) and hence still preserve much of the structure of a $\Cstar$-algebra. We call a system $(A_n,\rho_{n+1,n})_n$  of $\Cstar$-algebras with c.p.c.\ connecting maps 
a \emph{c.p.c.\ system}, 
and define the limit exactly as in the classical setting, 
except now the induced maps $\rho_n\colon A_n\to \prod_m A_m/\bigoplus_m A_m$ are only c.p.c., and the limit $\overline{\bigcup_n\rho_n(A_n)}$ is only a closed self-adjoint subspace.

A key source of examples of c.p.c.\ systems come from c.p.c.\ approximations of (separable) nuclear $\Cstar$-algebras. From \cite{CE78,Kir77}, we know that a (separable) $\Cstar$-algebra $A$ is nuclear precisely when it admits a system of c.p.c.\ approximations consisting of a sequence $(F_n)_n$ of finite-dimensional $\Cstar$-algebras and c.p.c.\ maps $A\xrightarrow{\psi_n}F_n\xrightarrow{\varphi_n}A$ such that $\varphi_n\circ\psi_n$ converges pointwise in norm to $\mathrm{id}_A$. The sequence $(F_n)_n$ together with the maps $\rho_{n+1,n}\coloneqq \psi_{n+1}\circ\varphi_n\colon F_n\to F_{n+1}$ form a c.p.c.\ system, which somehow encodes our original $\Cstar$-algebra $A$. In particular, the limit of the system is completely order isomorphic to $A$. A complete order isomorphism between closed self-adjoint subspaces of $\Cstar$-algebras is a completely positive, completely isometric map with completely positive inverse. These are extremely robust identifications. For example, a complete order isomorphism between $\Cstar$-algebras is automatically a $^*$-isomorphism, and so the complete order isomorphism class of a $\Cstar$-algebra determines its $^*$-isomorphism class.  

The question now is when does a given c.p.c.\ system actually encode a (nuclear) $\Cstar$-algebra? In this article, we answer this question by establishing necessary and sufficient conditions for the limit of a c.p.c.\ system to be completely order isomorphic to a (nuclear) $\Cstar$-algebra. These are built around the following definition. 
\begin{definition}\label{definition a}
We say a c.p.c.\ system $(A_n,\rho_{n+1,n})_n$  is \emph{$\Cstar$-encoding} if for any $k\geq 0$, $x\in A_k$, and $\varepsilon>0$, there exists $M>k$ so that for all $m>n,j>M$, 
\[
    \|\rho_{m,n}(\rho_{n,k}(x^*)\rho_{n,k}(x))-\rho_{m,j}(\rho_{j,k}(x^*)\rho_{j,k}(x))\|<\varepsilon.\]
\end{definition}
\noindent \cref{definition a} is built to guarantee that the bilinear map $\bigcup_n \rho_n(A_n)\times \bigcup_n \rho_n(A_n)\\ \to \overline{\bigcup_n \rho_n(A_n)}$, defined for each $k\geq 0$ and $\rho_k(x),\rho_k(y)\in \bigcup_n \rho_n(A_n)$  by 
\begin{align}\label{prod'}
    (\rho_k(x),\rho_k(y))\mapsto \lim_n\rho_n(\rho_{n,k}(x)\rho_{n,k}(y)), \tag{$\bullet$}
\end{align} 
gives a product on the limit $\overline{\bigcup_n\rho_n(A_n)}$. This product likely differs from the one on $\prod_m A_m/\bigoplus_m A_m$, but nonetheless, when equipped with this product, the limit \emph{is} a $\Cstar$-algebra (\cref{prop: encoding systems are cstar}). Moreover, the identity map gives a complete order isomorphism between this $\Cstar$-algebra and the original limit $\overline{\bigcup_n\rho_n(A_n)}$ (\cref{prop: manifesting are coi'}). The robustness of a complete order isomorphism guarantees that this $\Cstar$-algebra is unique up to $^*$-isomorphism (\cref{cor: limit is Cstar}), and hence we refer to it as the \emph{$\Cstar$-limit} of the system. If the $\Cstar$-algebras in a $\Cstar$-encoding system are finite-dimensional, then the $\Cstar$-limit is nuclear (\cref{thm: sato}). In this case we can actually say much more: When $\Cstar$-algebras in any given c.p.c.\ system are finite-dimensional, $\Cstar$-encoding is necessary and sufficient for the limit to be completely order isomorphic to a (nuclear) $\Cstar$-algebra (\cref{thm: big theorem}): 
\begin{theorem}\label{big theorem}
    Let $(F_n,\rho_{n+1,n})_n$ be a c.p.c.\ system with finite-dimensional $\Cstar$-algebras $F_n$. Then the following are equivalent. 
    \begin{enumerate}[label=\textnormal{(\roman*)}]
        \item The limit $\overline{\bigcup_n\rho_n(F_n)}$ is completely order isomorphic to a $\Cstar$-algebra.
        \item The limit $\overline{\bigcup_n\rho_n(F_n)}$ is completely order isomorphic to a nuclear $\Cstar$-algebra.
        \item $(F_n,\rho_{n+1,n})_n$  has a $\Cstar$-encoding subsystem. 
    \end{enumerate}
\end{theorem}
\noindent Since the limit of a c.p.c.\ system is completely order isomorphic to the limit of any subsystem (\cref{subsystem}), the subsystem criteria above is quite mild. It follows that the $\Cstar$-limit is unchanged by passing to a subsystem, and hence any $\Cstar$-algebra that is completely order isomorphic to the limit of a c.p.c.\ system with finite-dimensional $\Cstar$-algebras is $^*$-isomorphic to the $\Cstar$-limit of the system. 

As a converse to \cref{big theorem}, every separable nuclear $\Cstar$-algebra is $^*$-isomorphic to the $\Cstar$-limit of a $\Cstar$-encoding system -- in fact one coming from a system of c.p.c.\ approximations of the nuclear $\Cstar$-algebra (\cref{thm: main}):
\begin{theorem}\label{main}
Let $(A\xrightarrow{\psi_n}F_n\xrightarrow{\varphi_n}A)_n$ be a system of c.p.c.\ approximations of a separable nuclear $\Cstar$-algebra $A$. After possibly passing to a summable\footnote{See \cref{def: summable}. Any system of c.p.c.\ approximations admits a summable subsystem (\cite[Remark 3.2(ii)]{CW1}).} subsystem of approximations, the associated c.p.c.\ system $(F_n,\psi_{n+1}\circ\varphi_n)_n$ is $\Cstar$-encoding, and the  map $\Psi:A\to \prod_m F_m/\bigoplus_m F_m$ induced by the $\psi_m:A\to F_m$ gives a complete order isomorphism between $A$ and the limit of the $\Cstar$-encoding system. 
\end{theorem}

With Theorems \ref{big theorem} and \ref{main} combined, $\Cstar$-encoding systems give a notion of inductive systems which yield all nuclear $\Cstar$-algebras in their limits: 
\begin{theorem}\label{corollary d}
The following are equivalent for a separable $\Cstar$-algebra:
\begin{enumerate}
    \item $A$ is nuclear.
    \item $A$ is completely order isomorphic to the limit of a $\Cstar$-encoding system. 
    \item $A$ is $^*$-isomorphic to the $\Cstar$-limit of a $\Cstar$-encoding system. 
\end{enumerate}
\end{theorem}




We conclude by comparing $\Cstar$-encoding systems with NF systems (\cite[Definition 5.2.1]{BK97}), which are finite-dimensional c.p.c.\ systems with asymptotically multiplicative connecting maps, and with $\CPCstar$-systems (\cite[Definition 2.2]{CW1}), which are finite-dimensional c.p.c.\ systems with asymptotically order zero connecting maps. 
\begin{definition}
   Let $(F_n,\rho_{n+1,n})_n$  be a finite-dimensional c.p.c.\ system. 
   \begin{enumerate}[label=\textnormal{(\roman*)}]
\item We say $(F_n,\rho_{n+1,n})_n$ is  \emph{NF} if for any $k\geq 0$, $x,y\in F_k$, and $\varepsilon>0$, there exists $M>k$ so that for all $m>n>M$, 
\[
    \|\rho_{m,k}(x)\rho_{m,k}(y)-\rho_{m,n}(\rho_{n,k}(x)\rho_{n,k}(y))\|<\varepsilon.\]
\item We say $(F_n,\rho_{n+1,n})_n$ is  \emph{$\CPCstar$} if for any $k\geq 0$, $x,y\in F_k$, and $\varepsilon>0$, there exists $M>k$ so that for all $m>n,j>M$, 
\[
    \|\rho_{m,k}(x)\rho_{m,k}(y)-\rho_{m,j}(1_{F_j})\rho_{m,n}(\rho_{n,k}(x)\rho_{n,k}(y))\|<\varepsilon.\]
\item We say $(F_n,\rho_{n+1,n})_n$ is   \emph{$\Cstar$-encoding}\footnote{Here we use an equivalent formulation of \cref{definition a} for easier comparison (see \cref{same definitions}).} if for any $k\geq 0$, $x,y\in F_k$, and $\varepsilon>0$, there exists $M>k$ so that for all $m>n,j>M$, 
\[
    \|\rho_{m,j}(\rho_{j,k}(x)\rho_{j,k}(y))-\rho_{m,n}(\rho_{n,k}(x)\rho_{n,k}(y))\|<\varepsilon.\]
   \end{enumerate}

\end{definition}
 $\Cstar$-encoding systems immediately generalize NF systems and, by \cite[Proposition 2.7]{CW1} and \cref{big theorem}, they generalize  $\CPCstar$-systems as well. 
\cref{corollary d} should be compared with \cite[Theorem 5.2.2]{BK97}, which says that a separable $\Cstar$-algebra is nuclear and quasidiagonal if and only if it is isomorphic to the limit of an NF system. Quasidiagonal nuclear $\Cstar$-algebras form a well-studied class, but this class excludes many important nuclear $\Cstar$-algebras including the Toeplitz algebra, the Cuntz algebras, and more generally any $\Cstar$-algebra containing a proper isometry.  
 \cref{corollary d} should also be compared with \cite[Theorem C]{CW1}, which says that a separable $\Cstar$-algebra is nuclear if and only if it is isomorphic to the $\Cstar$-limit of a $\CPCstar$-system.  That means $\CPCstar$-systems and $\Cstar$-encoding systems both capture all nuclear $\Cstar$-algebras in their limits. But the correspondences in these theorems go deeper than that between nuclear (quasidiagonal) $\Cstar$-algebras and limits of $\Cstar$-encoding/NF/$\CPCstar$-systems. It turns out that any $\Cstar$-encoding system $(F_n,\rho_{n+1,n})_n$ with finite-dimensional $\Cstar$-algebras gives rise to a system of c.p.c.\ approximations $(A\xrightarrow{\psi_n}F_n\xrightarrow{\varphi_n}A)_n$ of its nuclear $\Cstar$-limit $A$, and when the system is moreover NF (resp.~$\CPCstar$), the $\psi_n$ are approximately (resp.~order zero) (\cref{encoding to approx}). Conversely, 
any system $(A\xrightarrow{\psi_n}F_n\xrightarrow{\varphi_n}A)_n$ of c.p.c.\ approximations of a nuclear $\Cstar$-algebra $A$ has a summable subsystem of approximations which gives rise to a $\Cstar$-encoding system (as in \cref{main}). Furthermore, if (and only if) the $\psi_n$ are approximately multiplicative (resp.~order zero), then (after possibly passing to a further subsystem) the $\Cstar$-encoding system is NF. 

The asymptotic multiplicativity and asymptotic order zero assumptions of NF and $\CPCstar$-systems carry significant structure, giving these systems more leverage to capture structural 
and tracial information of the $\Cstar$-limit. On the other hand, this structure often makes these assumptions difficult to satisfy. From \cite{BK97,BCW,CW1,WZ09} we know that for any separable nuclear and quasidiagonal (resp.\ nuclear) $\Cstar$-algebra $A$ \emph{there exists} a system $(A\xrightarrow{\psi_n}F_n\xrightarrow{\varphi_n}A)_n$ of c.p.c.\ approximations with $(\psi_n)_n$ approximately multiplicative (resp.\  order zero) so that the induced c.p.c.\ system $(F_n,\psi_{n+1}\circ\varphi_n)_n$ is NF (resp.\ $\CPCstar$). However, explicit examples of such systems of approximations
are much harder to come by than systems of approximations with just c.p.c.\ maps. Moreover, many classic constructions of systems of c.p.c.\ approximations are generally neither asymptotically multiplicative nor approximately order zero and will therefore induce neither NF nor $\CPCstar$-systems. This includes the usual systems of approximations for most amenable group $\Cstar$-algebras built from F{\o}lner sequences (see \cref{sect: Folner Z}). 
On the other hand, \emph{any} system of c.p.c.\ approximations, including these built from F{\o}lner sequences, induces a $\Cstar$-encoding system. 

\bigskip

\noindent \text{\bf Acknowledgements:} The author is grateful Wilhelm Winter and Jamie Gabe for many enlightening conversations and to Joachim Cuntz for his helpful feedback.

\renewcommand*{\thetheorem}{\roman{theorem}}
\numberwithin{theorem}{section}
\renewcommand*{\thecorollary}{\roman{theorem}}
\numberwithin{corollary}{section}
\renewcommand*{\thedefinition}{\roman{definition}}
\numberwithin{definition}{section}

\section{Preliminaries}\label{sect: prelim}
For a $\Cstar$-algebra $A$, we denote its set of positive elements by $A_+$, its closed unit ball by $A^1$, and the intersection of these two sets by $A_+^1$. 

Given a sequence $(A_n)_n$ of $\Cstar$-algebras, we denote by $\prod_n A_n$ the space of norm bounded sequences $(a_n)_n$ with $a_n\in A_n$ for all $n$; we denote by $\bigoplus_n A_n$ the two-sided closed ideal of $\prod_n A_n$ consisting of sequences converging to 0 in norm; and we denote the quotient $\Cstar$-algebra\footnote{ In the literature ``$A_\infty$'' is usually used to denote the sequence algebra $\ell^{\infty}(\mathbb{N}, A)/c_0(\mathbb{N},A)$ of a given $\Cstar$-algebra $A$. This agrees with our notation when $A_n=A$ for all $n$.} by
\begin{align*}
    A_\infty\coloneqq\textstyle{\prod_m A_m/\bigoplus_m A_m}.
\end{align*}
For $(a_n)_n\in \prod_n A_n$, we write $[(a_n)_n]$ for its image in this quotient.  

\begin{definition}\label{def: cp}
Let $A$ and $B$ be $\Cstar$-algebras and $X\subset A$ and $Y\subset B$ self adjoint subspaces. We say a $^*$-linear map $\theta\colon X\to Y$ is \emph{positive} if $\theta(x)\geq 0$ for all $x\geq 0$ and \emph{completely positive} (c.p.)\ if this holds for all matrix amplifications $\theta^{(r)}\colon \M_r(X)\to \M_r(Y)$. If $\theta$ is c.p.\ and $\sup_{r\geq 1}\|\theta^{(r)}\|=\|\theta\|\leq 1$ we call it \emph{completely positive and completely contractive} (abbreviated c.p.c.);  if $\theta^{(r)}$ is isometric for all $r\geq 1$, we call it \emph{completely isometric}. If $\theta$ is c.p.\ and completely isometric with c.p.\ inverse $\theta^{-1}\colon \theta(X)\to A$, we say it is a \emph{complete order embedding}, and when it is moreover surjective, we say it is a \emph{complete order isomorphism}.
\end{definition}

 \begin{remarks}\label{rmk: surj coi}
(i) If the domain of a completely isometric c.p.\ map is a $\Cstar$-algebra, then its inverse is automatically c.p., and it is automatically a complete order embedding 
 (see \cite[Remark 1.7(ii)]{CW1}).

(ii) Despite not being a $^*$-homomorphism, a complete order isomorphism is extremely robust. 
For instance, a complete order isomorphism $\theta\colon A\to B$ between $\Cstar$-algebras is automatically a $^*$-isomorphism (see for instance \cite[Theorem II.6.9.17]{Bla06}). Even if $\theta$ is only a complete order embedding, if there is another complete order embedding $\psi:C\to B$ from a $\Cstar$-algebra $C$ with $\psi(C)=\theta(A)$, then $\theta^{-1}\circ \psi:C\to A$ is a complete order isomorphism and hence a $^*$-isomorphism. 

(iii) Though we make no assumputions on units, we note that when $\theta$ is unital (i.e., u.c.p.), $\theta(A)\subset B$ is an operator subsystem and this coincides with the usual terminology for a complete order isomorphism. 
 \end{remarks}

Throughout the article, we will utilize a consequence of Stinespring's Dilation Theorem which is proved in \cite[Lemma 3.5]{KW} for positive elements. Since the proof deals only with the squares of positive elements, the exact same proof carries through for self-adjoint elements, and we state here the version we will use for easy reference. 

\begin{lemma}[{\cite[Lemma 3.5]{KW}}]\label{lem: Stinespring}
Let $A$ and $B$ be $\Cstar$-algebras, $a\in A$ self-adjoint, and $\eta>0$. If $A\xrightarrow{\psi}B\xrightarrow{\varphi}A$ are c.p.c.\ and $\|\varphi(\psi(a^i))-a^i\|<\eta^2/3$ for $i\in \{1,2\}$, then for all $b\in B$,
\begin{align}\label{eq: Stinespring}
    \|\varphi(\psi(a)b)-\varphi(\psi(a))\varphi(b)\|<\eta\|b\|.
\end{align}
\end{lemma}

Finally we recall from \cite{CE78,Kir77} that a separable $\Cstar$-algebra is nuclear if and only if it admits a system of c.p.c.\ approximations in the following sense:
\begin{definition}
    Let $A$ be a separable $\Cstar$-algebra. 
A \emph{system of c.p.c.\ approximations}  $(A\xrightarrow{\psi_n}F_n\xrightarrow{\varphi_n}A)_n$ of $A$ consists of a sequence $(F_n)_n$ of finite-dimensional $\Cstar$-algebras together with c.p.c.\ maps $\psi_n\colon A\to F_n$ and $\varphi_n\colon F_n\to A$ for all $n\in \mathbb{N}$ so that for each $a\in A$
\[\lim_n\|\varphi_n(\psi_n(a))-a\|=0.\]
\end{definition}

\section{Encoding a $\Cstar$-algebra into a c.p.c.\ system}\label{sect: abstract}

The aim of this section is to provide a condition on a system $A_0\xrightarrow{\rho_{1,0}}A_1\xrightarrow{\rho_{2,1}}A_2\longrightarrow\hdots$ of $\Cstar$-algebras with c.p.c.\ connecting maps, which guarantees a $\Cstar$-structure on the limit.

\begin{definition}\label{def: cpc system}
A \emph{c.p.c.\ system} $(A_n,\rho_{n+1,n})_n$ consists of a sequence $(A_n)_n$ of $\Cstar$-algebras along with a family of c.p.c.\ maps $\{\rho_{n+1,n}\colon A_n\to A_{n+1}\}_n$. 
 For $m>n\geq 0$, we set $\rho_{m,n}\coloneqq\rho_{m,m-1}\circ\hdots\circ\rho_{n+1,n}$ and $\rho_{n,n}\coloneqq \mathrm{id}_{A_n}$.  
When all the $\Cstar$-algebras are finite-dimensional, we call the system \emph{finite-dimensional}. 

Given a c.p.c.\ system $(A_n,\rho_{n+1,n})_n$ we define c.p.c.\ maps $\rho_n\colon A_n\to A_\infty$ by $\rho_n(x)=[(\rho_{m,n}(x))_{m>n}]$. We denote the closed self-adjoint subspace $\overline{\bigcup_n \rho_n(A_n)}\subset A_\infty$ by $\Lima$ as in \cite[Definition 2.1]{CW1} and call it the \emph{limit} of the system.

Given a c.p.c.\ system $(A_n, \rho_{n+1,n})$, a \emph{c.p.c.\ subsystem} of $(A_n, \rho_{n+1,n})$ is a c.p.c.\ system  $(A_{n_j}, \rho_{n_{j+1},n_j})_j$ where $(n_j)_j$ is a strictly increasing sequence in $\mathbb{N}$.
\end{definition}
Note that for a c.p.c.\ system $(A_n,\rho_{n+1,n})_n$ and any $r\geq 1$, we have 
    \begin{align}\label{Mr}
    \textstyle \M_r\Big(\overline{\bigcup_n \rho_n(A_n)}\Big)=\overline{\textstyle{\bigcup_n} \rho^{(r)}_n(\M_r(A_n))}\subset \M_r(A_\infty).\end{align}

    The following lemma tells us that the limit is unchanged if we replace a c.p.c.\ system with a subsystem. 

\begin{lemma}\label{subsystem}
Let $(A_n, \rho_{n+1,n})$ be a c.p.c.\ system, $(n_j)_j$ a strictly increasing sequence in $\mathbb{N}$, and $(A_{n_j}, \rho_{n_{j+1},n_j})_j$ the subsystem. Set $A_\infty= \sfrac{\prod_n A_n}{\bigoplus_n A_n}$ and $\hat{A}_\infty= \sfrac{\prod_j A_{n_j}}{\bigoplus_j A_{n_j}}$. Let $\rho_n:A_n\to A_\infty$ and $\hat{\rho}_{n_j}:A_{n_j}\to \hat{A}_\infty$ be the induced maps and $\Lima$ and $\Limaj$ the limits as in \cref{def: cpc system}. 
Then the surjective $^*$-homomorphism $\hat{\pi}:A_\infty\to \hat{A}_\infty$ induced by the natural surjection $\pi:\prod_n A_n\to \prod_j A_{n_j}$ restricts to a complete order isomorphism $\Lima\to \Limaj$. 
\end{lemma}

\begin{proof}
We aim to show that $\theta\coloneqq \hat{\pi}|_{\Lima}$ is a completely isometric c.p.\ map with c.p.\ inverse.
Since $\hat{\pi}$ is a $^*$-homomorphism, $\theta$ is c.p.c. We first show that it is  isometric. 
Since the maps are coherent,  $\bigcup_j \rho_{n_j}(A_{n_j})$ is dense in $\Lima$, and so it suffices to check that $\theta$ is isometric on $\rho_{n_k}(x)$ for fixed $n_k\geq 0$ and $x\in A_{n_k}$. Since $\theta\circ\rho_{n_k}(x)=\hat{\rho}_{n_k}(x)$, that amounts to showing that 
$\|\rho_{n_k}(x)\|=\|\hat{\rho}_{n_k}(x)\|$.  Since  $\rho_{m,n_k}$ is contractive for all $m>n_k$,  $(\|\rho_{m,n_k}(x)\|)_{m>k})$ is a decreasing sequence in $\mathbb{R}$ converging to $\|\rho_{n_k}(x)\|$, and  $(\|\rho_{n_j,n_k}(x)\|)_{j>k})$ is a decreasing subsequence of $(\|\rho_{m,n_k}(x)\|)_{m>k})$ converging to $\|\hat{\rho}_{n_k}(x)\|$. Hence the limits coincide, and thus the map is isometric. 
For $r\geq 1$, we know from \eqref{Mr}, that $\bigcup_j \rho^{(r)}_{n_j}(\M_r(A_{n_j}))$ is dense in $\M_r(\Lima)$ and each $\rho^{(r)}_{m,n}$ is still contractive. The same argument then shows that $\theta^{(r)}$ is isometric, and so $\theta$ is completely isometric. 

It remains to check that $\theta^{-1}$ is c.p. We define a c.p.c.\ split $\sigma: \prod_j A_{n_j}\to \prod_n A_n$ of $\pi$ by 
\begin{align*}
(x_{n_j})_j \mapsto (0,...,0, x_{n_0}, \rho_{n_0+1,n_0}(x_{n_0}),..., \rho_{n_1-1,n_0}(x_{n_0}), x_{n_1},\rho_{n_1+1,n_1}(x_{n_1}),...).
\end{align*} 
Let $\hat{\sigma}:\hat{A}_\infty\to A_\infty$ be the induced c.p.c.\ map. Then for each $j\geq 0$, we have $\hat{\rho}_{n_j}=\hat{\pi}\circ \rho_{n_j}=\theta\circ\rho_{n_j}$ and $\rho_{n_j}=\hat{\sigma}\circ\hat{\rho}_{n_j}$, i.e.,  
$\hat{\sigma}(\hat{\rho}_{n_j}(x))=\rho_{n_j}(x)$ and $\theta(\rho_{n_j}(x))=\hat{\rho}_{n_j}(x)$ for all $j\geq 0$ and $x\in A_{n_j}$. Since $\bigcup_j \rho_{n_j}(A_{n_j})$ is dense in $\Lima$ and $\bigcup_j \hat{\rho}_{n_j}(A_{n_j})$ is dense in $\Limaj$, this shows that $\hat{\sigma}|_{\Limaj}=\theta^{-1}$, and hence $\theta^{-1}$ is c.p.
\end{proof}

Now we are ready for our main definition. The presentation here looks a little different than \cref{definition a}, but as we shall see in \cref{same definitions}, the two definitions are equivalent. 

\begin{definition}\label{def: encoding}
We say a c.p.c.\ system $(A_n,\rho_{n+1,n})_n$ 
 is \emph{$\Cstar$-encoding} if for any $k\geq 0$, $x,y\in A_k$, and $\varepsilon>0$, there exists $M>k$ so that for all $m>n,j>M$, 
\begin{align*}
    \|\rho_{m,n}(\rho_{n,k}(x)\rho_{n,k}(y))-\rho_{m,j}(\rho_{j,k}(x)\rho_{j,k}(y))\|&<\varepsilon. 
    \end{align*}
\end{definition}

\begin{example}\label{prop: asym mult is encoding}
Recall from \cite[Definition 2.1.1]{BK97} that a c.p.c.\ system $(A_n,\rho_{n+1,n})_n$ is \emph{asymptotically multiplicative} if for any $k\geq 0$, $x,y\in A_k$, and $\varepsilon>0$, there exists $M>0$ so that for all $m>n>M$,
\begin{align*}
    \|\rho_{m,n}(\rho_{n,k}(x)\rho_{n,k}(y))-\rho_{m,k}(x)\rho_{m,k}(y)\|<\varepsilon.
\end{align*}
A finite-dimensional asymptotically multiplicative c.p.c.\ system is an \emph{NF system}. All asymptotically multiplicative c.p.c.\ systems, including NF systems, are $\Cstar$-encoding. 
\end{example}



Next, we give a lemma that will facilitate the associativity of our desired product.

\begin{lemma}\label{ass}
Let $(A_n,\rho_{n+1,n})$ be a $\Cstar$-encoding system. For any $k\geq 0, x,y,z\in A_k$, and $\varepsilon>0$, there exists an $M>k$ so that for all $m>n>j>M$
\begin{align}
    \|\rho_{m,j}\big(\rho_{j,k}(x)\rho_{j,k}(y)\rho_{j,k}(z)\big)-\rho_{m,n}\big(\rho_{n,k}(x)\rho_{n,j}\big(\rho_{j,k}(y)\rho_{j,k}(z)\big)\big)\|<\varepsilon.
\end{align}
    In particular, for any $k\geq 0$, $x,y,z\in A_k$, and $\varepsilon>0$, there exists $M>k$ so that for all $m>n>j>M$, 
    \begin{align}
    \|\rho_{m,n}\big(\rho_{n,j}\big(\rho_{j,k}(x)\rho_{j,k}(y)\big)\rho_{n,k}(z)\big) -\rho_{m,n}\big(\rho_{n,k}(x)\rho_{n,j}\big(\rho_{j,k}(y)\rho_{j,k}(z)\big)\big)\|&<\varepsilon.\label{eqn: asym associative}
    \end{align}
\end{lemma}

The following argument is essentially the proof of associativity in \cite[Theorem 3.1]{CE77}. The author is grateful to Jamie Gabe for suggesting this approach. 

\begin{proof}
First, we start with a generic matrix amplification fact that is well-known to experts: Let $A$ be a $\Cstar$-algebra and $u=\begin{pmatrix} u_{11} & u_{12}\\  u_{12}^* & u_{22}\end{pmatrix}\in \M_2(A)_+$. Then 
\begin{align}\label{matrix trick}
        \|u_{12}\|\leq \|u_{11}\|^{1/2}\big(\|u_{11}\|^{1/2}+\|u_{22}\|^{1/2}\big). 
\end{align}
To see this, let $v=\begin{pmatrix} v_{11} & v_{12}\\  v_{12}^* & v_{22}\end{pmatrix}\in \M_2(A)$ be a self-adjoint element such that 
\begin{align*}
    u=v^2=\begin{pmatrix} v_{11}^2+v_{12}v_{12}^* & v_{11}v_{12}+v_{12}v_{22}\\  v_{12}^*v_{11}+v_{22}v_{12}^* & v_{12}^*v_{12}+v_{22}^2\end{pmatrix}.
\end{align*} 
Then
\begin{align*}
    \|u_{12}\|&=\|v_{11}v_{12}+v_{12}v_{22}\|\\
    &\leq \|v_{12}\|\big(\|v_{11}\|+\|v_{22}\|\big)\\
    &\leq \|v_{11}^2+v_{12}v_{12}^*\|^{1/2}\big(\|v_{11}^2+v_{12}v_{12}^*\|^{1/2} +\|v_{12}^*v_{12}+v_{22}^2\|^{1/2}\big)\\
    &= \|u_{11}\|^{1/2}\big(\|u_{11}\|^{1/2}+\|u_{22}\|^{1/2}\big), 
\end{align*}
which establishes \eqref{matrix trick}.

    Now, set $k\geq 0, x,y,z\in A_k^1$, and $\varepsilon>0$. 
    For each $j>k$, set 
    \begin{align*}
    x_j&\coloneqq \rho_{j,k}(x)\\
        a_j&\coloneqq 
        \rho_{j,k}(y)\rho_{j,k}(z)
        \\
        d_j&\coloneqq \begin{pmatrix}
            0 & x_j\\ x_j^* & a_j\end{pmatrix}\in \M_2(A_j).
    \end{align*}
    By the Kadison-Schwarz inequality, we have for all $n>j>k$
        \begin{align*}
        \rho_{n,j}^{(2)}(d_j^*d_j)&= \begin{pmatrix} \rho_{n,j}(x_jx_j^*) & \rho_{n,j}(x_ja_j)\\ \rho_{n,j}(a_j^*x_j^*) & \rho_{n,j}(x_j^*x_j+a_j^*a_j)            \end{pmatrix}\\
        &\geq \rho_{n,j}^{(2)}(d_j)^*\rho_{n,j}^{(2)}(d_j)\\&=\begin{pmatrix} \rho_{n,j}(x_j)\rho_{n,j}(x_j)^* & \rho_{n,j}(x_j)\rho_{n,j}(a_j)\\ \rho_{n,j}(a_j)^*\rho_{n,j}(x_j)^* & \rho_{n,j}(x_j)^*\rho_{n,j}(x_j)+\rho_{n,j}(a_j)^*\rho_{n,j}(a_j)   
        \end{pmatrix}.
    \end{align*}
    Then for all $m>n>j>k$, we have 
    \begin{align*}
       \rho_{m,n}^{(2)}\big(\rho_{n,j}^{(2)}(d_j^*d_j)-\rho_{n,j}^{(2)}(d_j)^*\rho_{n,j}^{(2)}(d_j)\big)=\rho_{m,j}^{(2)}(d_j^*d_j)-\rho_{m,n}^{(2)}\big(\rho_{n,j}^{(2)}(d_j)^*\rho_{n,j}^{(2)}(d_j)\big)\geq 0.
    \end{align*}
    Writing
    \begin{align}\label{u}
        u=\begin{pmatrix} u_{11} & u_{12}\\  u_{12}^* & u_{22}\end{pmatrix}\coloneqq \rho_{m,j}^{(2)}(d_j^*d_j)-\rho_{m,n}^{(2)}\big(\rho_{n,j}^{(2)}(d_j)^*\rho_{n,j}^{(2)}(d_j)\big),
    \end{align} 
    we have
    \begin{align*}
        u_{11}&=\rho_{m,j}(x_jx_j^*)-\rho_{m,n}\big(\rho_{n,j}(x_j)\rho_{n,j}(x_j)^*\big) \\
        &=\rho_{m,j}(\rho_{j,k}(x)\rho_{j,k}(x)^*)-\rho_{m,n}\big(\rho_{n,j}(\rho_{j,k}(x))\rho_{n,j}(\rho_{j,k}(x))^*\big) \\
        &=\rho_{m,j}(\rho_{j,k}(x)\rho_{j,k}(x)^*)-\rho_{m,n}(\rho_{n,k}(x)\rho_{n,k}(x)^*) \\
        u_{12}&=\rho_{m,j}(x_ja_j)-\rho_{m,n}\big(\rho_{n,j}(x_j)\rho_{n,j}(a_j)\big)\\
        &=\rho_{m,j}\big(\rho_{j,k}(x)\rho_{j,k}(y)\rho_{j,k}(z)\big)-\rho_{m,n}\big(\rho_{n,j}(\rho_{j,k}(x))\rho_{n,j}\big(\rho_{j,k}(y)\rho_{j,k}(z)\big)\big)\\
        &=\rho_{m,j}\big(\rho_{j,k}(x)\rho_{j,k}(y)\rho_{j,k}(z)\big)-\rho_{m,n}\big(\rho_{n,k}(x)\rho_{n,j}\big(\rho_{j,k}(y)\rho_{j,k}(z)\big)\big)
        \end{align*}
        and
        \begin{align*}
        \|u_{22}\|&=\|\rho_{m,j}(x_j^*x_j+a_j^*a_j)-\rho_{m,n}\big(\rho_{n,j}(x_j)^*\rho_{n,j}(x_j)+\rho_{n,j}(a_j)^*\rho_{n,j}(a_j)\big)\|\leq 4.
    \end{align*}
    
    Choose $\eta>0$ so that $\eta^{1/2}(\eta^{1/2}+2)<\sfrac{\varepsilon}{2}$, and use \cref{def: encoding} to choose $M>k$ so that for all $m>n,j>M$
    \begin{align}\label{x estimate}
        \|\rho_{m,j}(\rho_{j,k}(x)\rho_{j,k}(x)^*)-\rho_{m,n}(\rho_{n,k}(x)\rho_{n,k}(x)^*)\|<\eta.
    \end{align}
Then for any $m>n>j>M$, we apply  \eqref{matrix trick} to our $u$ in \eqref{u} for this $m>n>j$ and use the estimate from \eqref{x estimate} to get 
    \begin{align*}
         &\|\rho_{m,j}\big(\rho_{j,k}(x)\rho_{j,k}(y)\rho_{j,k}(z)\big)-\rho_{m,n}\big(\rho_{n,k}(x)\rho_{n,j}\big(\rho_{j,k}(y)\rho_{j,k}(z)\big)\big)\|\\
         &=\|u_{12}\|\\
&\leq \|u_{11}\|^{1/2}\big(\|u_{11}\|^{1/2}+\|u_{22}\|^{1/2}\big)\\
&<\eta^{1/2}(\eta^{1/2}+2)\\
&<\sfrac{\varepsilon}{2}.
    \end{align*}

    A similar estimate (or simply taking adjoints) shows that there exists $M'>k$ so that for all $m>n>j>M'$, 
    \begin{align*}
    \|\rho_{m,n}\big(\rho_{n,j}\big(\rho_{j,k}(x)\rho_{j,k}(y)\big)\rho_{n,k}(z)\big) -\rho_{m,n}\big(\rho_{n,j}\big(\rho_{j,k}(x)\rho_{j,k}(y)\big)\rho_{n,k}(z)\big)\|&<\sfrac{\varepsilon}{2},
    \end{align*}
    and together, these establish \eqref{eqn: asym associative}.
\end{proof}

\begin{remark}\label{same definitions}
    By instead setting $a_j\coloneqq \rho_{j,k}(y)$, the preceding proof also shows that \cref{def: encoding} is indeed equivalent to \cref{definition a}. 
\end{remark}

For the sake of easy reference in the proofs of \cref{prop: encoding systems are cstar} and \cref{prop: manifesting are coi'}, we highlight how some of the conditions from \cref{def: encoding} 
look in the quotient $A_\infty$. 

\begin{lemma}\label{prop: observations in infty}
The sequence $\big(\rho^{(r)}_n(\rho^{(r)}_{n,k}(x)\rho^{(r)}_{n,k}(y))\big)_n$ converges in $\M_r\big(\Lima\big)$ for any $r\geq 1$, $k\geq 0$ and $x,y\in \M_r(A_k)$, and 
\begin{enumerate}
    \item $\|\lim_n\rho^{(r)}_n\big(\rho^{(r)}_{n,k}(x)\rho^{(r)}_{n,k}(y)\big)\|\leq \|\rho^{(r)}_k(x)\|\|\rho^{(r)}_k(y)\|$ \label{it: at infty 3} and
    \item $\|\lim_n\rho^{(r)}_n\big(\rho^{(r)}_{n,k}(x)^*\rho^{(r)}_{n,k}(x)\big)\|=\|\rho^{(r)}_k(x)\|^2$.\label{it: at infty 4}
\end{enumerate}
\end{lemma}

\begin{proof}
For any $k\geq 0$ and $x,y\in A_k$, the sequence $\rho_n(\rho_{n,k}(x)\rho_{n,k}(y))$ is Cauchy by \cref{def: encoding} and hence converges in $\Lima$. Applying this coordinate-wise shows the same for $r\geq 1$. 

For \eqref{it: at infty 3}, we claim first that for any fixed $k\geq 0$, $r\geq 1$, $z\in \M_r(A_k)$, and $\varepsilon_0>0$, there exists $M>k$ so that for all $m>n>M$, 
    \begin{align}\label{eqn: everything else}
    \|\rho^{(r)}_{m,n}\big(\rho^{(r)}_{n,k}(z)^*\rho^{(r)}_{n,k}(z)\big)\|&<\|\rho^{(r)}_{m,k}(z)\|^2 +\varepsilon_0. 
    \end{align}
Indeed, choose $\eta_0>0$ so that $2\|z\|\eta_0+\eta_0^2<\varepsilon_0$. Since $(\|\rho^{(r)}_{n,k}(z)\|)_{n>k}$ is bounded and decreasing, it converges and hence there exists $M>k$ so that $\|\rho^{(r)}_{n,k}(z)\|< \|\rho^{(r)}_{m,k}(z)\|+\eta_0$ for all $m>n>M$. Then for all $m>n>M$ we have 
    \begin{align*}
       \|\rho^{(r)}_{m,n}\big(\rho^{(r)}_{n,k}(z)^*\rho^{(r)}_{n,k}(z)\big)\|&\leq \|\rho^{(r)}_{n,k}(z)^*\rho^{(r)}_{n,k}(z)\|\\
        &=\|\rho^{(r)}_{n,k}(z)\|^2\\
        &<\big(\|\rho^{(r)}_{m,k}(z)\|+\eta_0\big)^2\\
        &=\|\rho^{(r)}_{m,k}(z)\|^2+2\|\rho^{(r)}_{m,k}(z)\|\eta_0 +\eta_0^2\\
        &\leq \|\rho^{(r)}_{m,k}(z)\|^2+2\|z\|\eta_0 +\eta_0^2\\
        &<\|\rho^{(r)}_{m,k}(z)\|^2+\varepsilon_0.
   \end{align*}
   This establishes \eqref{eqn: everything else}.

Now, set $\varepsilon>0$, $r\geq 1$, $k\geq 0$, and $x,y\in  \M_r(A_k)$. 
 Choose $\eta>0$ so that $\|x\|\|y\|\eta+\eta^2< \varepsilon$. By \eqref{eqn: everything else}, we may choose $M>k$ so that for all $m>n>M$, 
\begin{align*}
    \|\rho^{(r)}_{m,n}\big(\rho^{(r)}_{n,k}(x)\rho^{(r)}_{n,k}(x)^*\big)\|&<\|\rho^{(r)}_{m,k}(x)^*\|^2 +\eta=\|\rho^{(r)}_{m,k}(x)\|^2 +\eta\ \text{ and }\\
    \|\rho^{(r)}_{m,n}\big(\rho^{(r)}_{n,k}(y)^*\rho^{(r)}_{n,k}(y)\big)\|&<\|\rho^{(r)}_{m,k}(y)\|^2 +\eta.
\end{align*}
Then applying 
a Cauchy-Schwarz type inequality for c.p.\ maps (a consequence of Stinespring's dilation theorem), 
we have for all $m>n>M$,
\begin{align*}
    \|\rho^{(r)}_{m,n}\big(\rho^{(r)}_{n,k}(x)\rho^{(r)}_{n,k}(y)\big)\|^2&\leq \|\rho^{(r)}_{m,n}\big(\rho^{(r)}_{n,k}(x)\rho^{(r)}_{n,k}(x)^*\big)\|\|\rho^{(r)}_{m,n}\big(\rho^{(r)}_{n,k}(y)^*\rho^{(r)}_{n,k}(y)\big)\|\\
    &< \big(\|\rho^{(r)}_{m,k}(x)\|^2 +\eta\big)\big(\|\rho^{(r)}_{m,k}(y)\|^2 +\eta\big)\\
    &\leq \|\rho^{(r)}_{m,k}(x)\|^2\|\rho^{(r)}_{m,k}(y)\|^2 +\|x\|\|y\|\eta +\eta^2\\
    &<\|\rho^{(r)}_{m,k}(x)\|^2\|\rho^{(r)}_{m,k}(y)\|^2 +\varepsilon. 
\end{align*}  
For (2) note that for any $k\geq 0$, $x\in A_k$, and $n\geq k$,
\begin{align*}
\|\rho^{(r)}_k(x)\|^2&=\|\rho^{(r)}_k(x)^*\rho^{(r)}_k(x)\|\\
&=\|\rho^{(r)}_n(\rho^{(r)}_{n,k}(x)^*)\rho^{(r)}_n(\rho^{(r)}_{n,k}(x))\|\\
&\leq \|\rho^{(r)}_n(\rho^{(r)}_{n,k}(x)^*\rho^{(r)}_{n,k}(x))\|,
\end{align*}
and so (2) follows from (1).
\end{proof}




\begin{proposition}\label{prop: encoding systems are cstar}
Let $(A_n,\rho_{n+1,n})$ be a $\Cstar$-encoding system with limit $\Lima$. Then there exists an associative bilinear map $\sbt\colon \Lima \times \Lima \to \Lima$, given on $\bigcup_n \rho_n(A_n)$ by
\begin{align*}
    \rho_k(x)\sbt \rho_k(y) & = \lim_n \rho_n(\rho_{n,k}(x) \rho_{n,k}(y))
\end{align*} 
 for each $k\geq 0$ and $x,y\in A_k$, 
so that when equipped with this product, $\Lima$ is a $\Cstar$-algebra with the same involution and norm inherited as a subspace of $A_\infty$. 
We denote this $\Cstar$-algebra by $\Cstar_{\sbt}(\Lima)$. 
\end{proposition}
Note that $\Cstar_{\sbt}(\Lima)$ is just the closed involutive Banach space $\Lima\subset A_\infty$ equipped with a product, $\sbt$, that likely does not agree with multiplication on $A_\infty$. We use the usual concatenation to denote multiplication in $A_\infty$ (i.e., $\bar{x}\bar{y}$) and $\sbt$ to denote the multiplication in $\Cstar_{\sbt}(\Lima)$ (i.e., $\bar{x}\sbt \bar{y}$). 

\begin{proof}
We begin by defining $\sbt\colon \bigcup_n \rho_n(A_n) \times \bigcup_n \rho_n(A_n) \to \Lima$. 
Since the union $\bigcup_n \rho_n(A_n)$ is nested, 
for any $\bar{x},\bar{y}\in \bigcup_n \rho_n(A_n)$, 
we may choose $k\geq 0$ and $x,y\in A_k$ so that $\bar{x}=\rho_k(x)$ and $\bar{y}=\rho_k(y)$, and so we define $\bar{x}\sbt \bar{y}$ by 
\begin{align*}
     \bar{x}\sbt \bar{y} & \coloneqq \lim_n \rho_n(\rho_{n,k}(x) \rho_{n,k}(y)).
\end{align*}
The limit exists by \cref{prop: observations in infty}. Moreover, for any lifts $(x_n)_n, (y_n)_n\in \prod_n A_n$ of $\bar{x}$ and $\bar{y}$ and any $\varepsilon>0$, there exists $M>k$ with \[\sup_{n>M}\|\rho_n(\rho_{n,k}(x)\rho_{n,k}(y))-\rho_n(x_ny_n)\|\leq \sup_{n>M} \|\rho_{n,k}(x)\rho_{n,k}(y)-x_ny_n\|<\varepsilon.\]
Hence the limit is independent of the choice of lifts of $\bar{x}$ and $\bar{y}$, and so we write
\begin{align}
   \rho_k(x)\sbt \rho_k(y)=\lim_n \rho_n(\rho_{n,k}(x) \rho_{n,k}(y))\label{eq: prod on dense}
\end{align}
for each  $k\geq 0$ and  $x,y\in A_k$. 
One readily checks that $\sbt$ is also bilinear.

From \cref{prop: observations in infty}\eqref{it: at infty 3} we know that for any $k\geq 0$ and $x,y\in A_k$
\begin{align}\label{there}
    \|\rho_k(x)\sbt\rho_k(y)\|\leq \|\rho_k(x)\|\|\rho_k(y)\|,
\end{align}
which implies $\sbt$ is bounded, and so we may extend it to a bounded bilinear map $\sbt\colon \Lima\times \Lima\to \Lima$ so that for all $\bar{x},\bar{y}\in \Lima$
\begin{align}\label{eqn: Bstar}
    \|\bar{x}\sbt \bar{y}\| 
    \leq 
    \|\bar{x}\|\|\bar{y}\|.
\end{align} 

Recall from \cite[Lemma 2.3]{CW1}\footnote{Although \cite[Lemma 2.3]{CW1}  was stated for finite-dimensional $\Cstar$-algebras, that played no role in the proof.} 
that for any $\bar{x}\in \Lima$ and any lift $(x_n)_n\in \prod_n A_n$ of $x$, we have $\bar{x}=\lim_n\rho_n(x_n)$. It follows that for any $\bar{x},\bar{y}\in \Lima$ we can express $\bar{x}\sbt\bar{y}$ as 
\begin{align}
    \bar{x}\sbt \bar{y} =\lim_n \rho_n(x_n)\sbt \rho_n(y_n), \label{bullet at closure}
\end{align}
where $(x_n)_n, (y_n)_n\in \prod_n A_n$ are  lifts of $\bar{x}$ and $\bar{x}$, respectively. 

To show that we have a product, it remains to check associativity. To that end, we first check that for fixed $k\geq 0$ and $x,y,z\in A_k$,
\begin{align*}(\rho_k(x)\sbt \rho_k(y))\sbt \rho_k(z)= \rho_k(x)\sbt(\rho_k(y)\sbt \rho_k(z)). 
\end{align*}
Since $\sbt$ is bounded, we can use \eqref{eq: prod on dense} to rewrite 
\begin{align}
(\rho_k(x)\sbt \rho_k(y))\sbt \rho_k(z)&= \left(\lim_n\rho_n(\rho_{n,k}(x)\rho_{n,k}(y))\right)\sbt \rho_k(z)\label{eq: lim 0}\\
&=\lim_n\big(\rho_n(\rho_{n,k}(x)\rho_{n,k}(y))\sbt\rho_k(z)\big)\ \text{ and}\nonumber\\ 
\rho_k(x)\sbt(\rho_k(y)\sbt \rho_k(z))&= \rho_k(z)\sbt \left(\lim_n\rho_n(\rho_{n,k}(y)\rho_{n,k}(z))\right)\label{eq: lim 1}\\
&=\lim_n\big(\rho_k(x)\sbt \rho_n(\rho_{n,k}(y)\rho_{n,k}(z)\big). \nonumber
\end{align}
Fix $\varepsilon>0$. Using \eqref{eq: lim 0} and \eqref{eq: lim 1}, we choose $N>k$ so that for all $n>N$, 
\begin{align}
\|\rho_n(\rho_{n,k}(x)\rho_{n,k}(y))\sbt\rho_k(z)-(\rho_k(x)\sbt \rho_k(y))\sbt \rho_k(z)\|&<\sfrac{\varepsilon}{5},\label{eq: lim 2} \ \text{ and}\\ 
\|\rho_k(x)\sbt \rho_n(\rho_{n,k}(y)\rho_{n,k}(z))-\rho_k(x)\sbt (\rho_k(y)\sbt \rho_k(z))\|&<\sfrac{\varepsilon}{5}\label{eq: lim 3}
\end{align}
By \eqref{eqn: asym associative}, we can choose $M>N$ so that for all $m>n>M$
\begin{align}\label{eq: lim 4}
     \left\|\rho_m\big(\rho_{m,n}\big(\rho_{n,k}(x)\rho_{n,k}(y)\big)\rho_{m,k}(z)\big) -\rho_m\big(\rho_{m,k}(x)\rho_{m,n}\big(\rho_{n,k}(y)\rho_{n,k}(z)\big)\big)\right\|<\sfrac{\varepsilon}{5}, 
\end{align}
Fix $n>M>N$. Using \eqref{eq: prod on dense} on $\rho_n(\rho_{n,k}(x)\rho_{n,k}(y))\sbt\rho_k(z)$ and\\ $\rho_k(x)\sbt \rho_n(\rho_{n,k}(y)\rho_{n,k}(z))$, we can choose $m>n$ so that 
\begin{align}
    \left\|\rho_m\big(\rho_{m,n}\big(\rho_{n,k}(x)\rho_{n,k}(y)\big)\rho_{m,k}(z)\big)-\rho_n(\rho_{n,k}(x)\rho_{n,k}(y))\sbt\rho_k(z)\right\|&<\sfrac{\varepsilon}{5},\label{eq: lim 5}
    \end{align}
    and
    \begin{align}
    \left\|\rho_m\big(\rho_{m,k}(x)\rho_{m,n}\big(\rho_{n,k}(y)\rho_{n,k}(z)\big)\big)-\rho_k(x)\sbt \rho_n(\rho_{n,k}(y)\rho_{n,k}(z))\right\|&<\sfrac{\varepsilon}{5}.\label{eq: lim 6}
\end{align}
For our fixed $m>n>M>N$, we use \eqref{eqn: Bstar} and bilinearity to estimate
\begin{align*}
     &\phantom{\overset{\eqref{eq: lim 2},\eqref{eq: lim 3}}{<}} \|(\rho_k(x)\sbt \rho_k(y))\sbt \rho_k(z)- \rho_k(x)\sbt(\rho_k(y)\sbt \rho_k(z))\|\\
    &\overset{\eqref{eq: lim 2},\eqref{eq: lim 3}}{<} \|\rho_n(\rho_{n,k}(x)\rho_{n,k}(y))\sbt\rho_k(z)-\rho_k(x)\sbt \rho_n(\rho_{n,k}(y)\rho_{n,k}(z))\| +  2(\sfrac{\varepsilon}{5})\\
   &\overset{\eqref{eq: lim 5},\eqref{eq: lim 6}}{<}  \|\rho_m\big(\rho_{m,n}\big(\rho_{n,k}(x)\rho_{n,k}(y)\big)\rho_{m,k}(z)\big) -\rho_m\big(\rho_{m,k}(x)\rho_{m,n}\big(\rho_{n,k}(y)\rho_{n,k}(z)\big)\big)\| \\ & \phantom{----}  +  4(\sfrac{\varepsilon}{5})\\
  &\overset{\phantom{--}\eqref{eq: lim 4}\phantom{--}}{<}  \varepsilon.
\end{align*}
Since $\varepsilon$ was arbitrary, this proves that \begin{align}\label{eq: ass on dense}
    (\rho_k(x)\sbt \rho_k(y))\sbt \rho_k(z)= \rho_k(x)\sbt(\rho_k(y)\sbt \rho_k(z)).
\end{align}

For general $\bar{x},\bar{y},\bar{z}\in \Lima$, let $(x_n)_n,(y_n)_n, (z_n)_n\in \prod_n A_n$ be respective lifts. Then as before, $\lim_n\rho_n(x_n) = \bar{x}$, $\lim_n\rho_n(y_n) =\bar{y}$, and $\lim_n\rho_n(z_n) =\bar{z}$. 
Then it follows from bilinearity, \eqref{eqn: Bstar}, and \eqref{eq: ass on dense} that 
\begin{align*}
    (\bar{x}\sbt\bar{y})\sbt\bar{z} &=\lim_n\big(\rho_n(x_n)\sbt\rho_n(y_n)\big)\sbt \rho_n(z_n)
    =\lim_n\rho_n(x_n)\sbt\big(\rho_n(y_n)\sbt \rho_n(z_n)\big) 
    =\bar{x}\sbt(\bar{y}\sbt\bar{z}). 
\end{align*}

Next we check that $(\Lima, \sbt)$ is a $^*$-algebra with respect to the $^*$-operation on $A_\infty$.
For $k\geq 0$ and $x,y\in A_k$, we have 
\begin{align*}
    \rho_k(y)^*\sbt \rho_k(x)^*&=\lim_n\rho_n\big(\rho_{n,k}(y)^*\rho_{n,k}(x)^*\big)=\lim_n\rho_n\big(\rho_{n,k}(x)\rho_{n,k}(y))\big)^*\\
    &=\left(\lim_n\rho_n\big(\rho_{n,k}(x)\rho_{n,k}(y)\big)\right)^*=(\rho_k(x)\sbt\rho_k(y))^*,
\end{align*}
and so by continuity $(\bar{x}\sbt\bar{y})^*=\bar{x}^*\sbt\bar{y}^*$ for all $\bar{x},\bar{y}\in \Lima$.

Hence  $(\Lima, \sbt)$ is a $^*$-algebra, and by \eqref{eqn: Bstar}, $(\Lima, \sbt, \|\cdot\|_{A_\infty})$ is moreover a Banach $^*$-algebra.
 \cref{prop: observations in infty}\eqref{it: at infty 4} tells us that the $\Cstar$-identity holds on the dense subspace $\bigcup_n\rho_n(A_n)$ of $\Lima$ with the $\sbt$-multiplication. 
 Again by continuity of $\sbt$, it follows that for all $\bar{x}\in \Lima$
\begin{align}\label{here} 
\|\bar{x}\|^2
=\|\bar{x}^*\sbt\bar{x}\|. 
\end{align}
Hence $\Lima$ equipped the product $\sbt$ and with the involution and norm from $A_\infty$ is a $\Cstar$-algebra. 
\end{proof}

As involutive Banach spaces, $\Lima$ and $\Cstar_{\sbt}(\Lima)$ are equal. It turns out that they are moreover completely order isomorphic via the identity map. 

\begin{proposition}\label{prop: manifesting are coi'}
Let $(A_n,\rho_{n+1,n})_n$ be a $\Cstar$-encoding system. Then the map
\begin{align*}
    \Theta\coloneqq\mathrm{id}_{\Lima}\colon \Cstar_{\sbt}(\Lima)\to A_\infty
\end{align*} is a complete order embedding. 
\end{proposition}

\begin{proof}
By Remark \ref{rmk: surj coi}(i), it suffices to show that it is c.p.\ and completely isometric. Fix $r\geq 1$. 
To see that $\Theta^{(r)}$ is isometric, it suffices to show that the norm on $\M_r(\Cstar_{\sbt}(\Lima))$ agrees with the norm on $\M_r(A_\infty)$.  
 Since $\Lima$ and $\Cstar_{\sbt}(\Lima)$ agree as sets, it follows from \eqref{Mr} that 
 \begin{align}
     \M_r(\Cstar_{\sbt}(\Lima))=\M_r(\Lima)=\overline{\textstyle{\bigcup_n} \rho^{(r)}_n(\M_r(A_n))}^{\|\cdot\|_{\M_r(A_\infty)}}.\label{eq: coi'}
 \end{align}
  In particular, $\M_r(\Cstar_{\sbt}(\Lima))$ is complete with respect to the norm on $\M_r(A_\infty)$, and so to show that these norms agree, it suffices to show that $\|\cdot\|_{\M_r(A_\infty)}$ is a $\Cstar$-norm for $\M_r(\Cstar_{\sbt}(\Lima))$ (meaning with respect to the $\sbt$-multiplication).
  Notice that on $\M_r(\Cstar_{\sbt}(\Lima))$, the product $\sbt$ satisfies an amplified version of \eqref{bullet at closure}, i.e., 
for any $\bar{x}, \bar{y}\in \Lima$ and any lifts $(x_n)_n,(y_n)_n\in \prod_n \M_r(A_n)$ lifts of $\bar{x}$ and $\bar{y}$, respectively,
\[\bar{x}\sbt\bar{y} =\lim_n\rho_n^{(r)}(x_n)\sbt\rho_n^{(r)}(y_n),\]
and moreover, this operation is also bounded as in \eqref{there}. 
Then, just as in \eqref{eqn: Bstar} and \eqref{here}, we can leverage continuity along with \cref{prop: observations in infty}(i) and (ii) to conclude that $\|\cdot\|_{A_\infty}$ defines a Banach algebra norm on $\M_r(\Cstar_{\sbt}(\Lima))$ which moreover satisfies the $\Cstar$-identity. Hence the norm on $\M_r(\Cstar_{\sbt}(\Lima))$ agrees with the norm on $\M_r(A_\infty)$, and $\Theta^{(r)}$ is isometric. 

Next we show that $\Theta^{(r)}$ is positive, i.e., for any fixed $\bar{x}\in \M_r(\Lima)$ we have $\bar{x}^*\sbt \bar{x}\in \M_r(A_\infty)_+$ . Let 
$(x_n)_n\in \prod_n\M_r(A_n)$ be a lift of $\bar{x}$. Then $\lim_n\rho_n^{(r)}(x_n)=\bar{x}$ and $\lim_n\rho_n^{(r)}(x_n)^*=\bar{x}^*$ and so
\[
\bar{x}^*\sbt \bar{x}=\lim_n \rho_n^{(r)}(x_n)^*\sbt \rho_n^{(r)}(x_n)=\lim_m \lim_n\rho_m^{(r)}\big(\rho_{m,n}^{(r)}(x_n)^*\rho_{m,n}^{(r)}(x_n)\big)\in \M_r(A_\infty)_+. \qedhere
\]
\end{proof}

It now follows from Remark \ref{rmk: surj coi}(ii), that $\Cstar_{\sbt}(\Lima)$ is unique up to $^*$- isomorphism in the following sense:

\begin{corollary}\label{cor: limit is Cstar}
Let $(A_n,\rho_{n+1,n})_n$ be a $\Cstar$-encoding 
system and $\Theta\colon \Cstar_{\sbt}(\Lima)\\ \to A_\infty$ as in \cref{prop: manifesting are coi'}. Then for any $\Cstar$-algebra $A$ and complete order embedding $\psi\colon A\to A_\infty$ with $\psi(A)=\Lima$, the map $\Theta^{-1}\circ \psi\colon A\to \Cstar_{\sbt}(\Lima)$ is a $^*$-isomorphism.
\end{corollary}

Because of the robustness of this correspondence, we are justified in also viewing $\Cstar_{\sbt}(\Lima)$ as the limit of the system. For the sake of clarity, we will refer to it as the $\Cstar$-limit.

\begin{definition}
For a $\Cstar$-encoding system $(A_n,\rho_{n+1,n})_n$, we call $\Cstar_{\sbt}(\Lima)$ from \cref{prop: encoding systems are cstar} the \emph{$\Cstar$-limit} of the system.
\end{definition}

Unlike in \cite{CW1}, we have no need to employ the element $e\coloneqq[(\rho_{n+1,n}(1_{A_n}))_n]\in A_\infty$ in our proofs. Nonetheless, if it happens to lie in the limit of a $\Cstar$-encoding system then it will automatically be the unit of the $\Cstar$-limit of the system. The proof is essentially the one given in \cite[Lemma 2.5(ii)]{CW1}; however since \cite[Lemma 2.5]{CW1} is stated for $\CPCstar$-systems and where $e$ is possibly not in $\Lima$, we give the proof here in our context.

\begin{proposition}\label{prop: e unit}
Let $(A_n,\rho_{n+1,n})_n$ be a $\Cstar$-encoding system with unital $\Cstar$-algebras $A_n$, and set $e\coloneqq[(\rho_{n+1,n}(1_{A_n}))_n]\in A_\infty$. 
If $e\in \Lima$, then $e$ is the unit for $\Cstar_{\sbt}(\Lima)$. 
\end{proposition}

\begin{proof}
Suppose $e\in \Lima$. Then by 
\cite[Lemma 2.3]{CW1}, 
\[e=\lim_n\rho_{n+1}(\rho_{n+1,n}(1_{A_n}))=\lim_n \rho_n(1_{A_n}). 
\] 
We claim that
$e- \bar{x} \in (A_\infty)_+$ for all self-adjoint $\bar{x}\in \Lima$ with $\|\bar{x}\|\leq 1$. 
 Indeed, using again \cite[Lemma 2.3]{CW1}, it suffices to prove the claim for $\bar{x}=\rho_k(x)$ for a fixed $k\geq 0$ and self-adjoint $x\in A_k$. 
Since $\|\rho_{n,k}(x)\|1_{A_n}\geq \rho_{n,k}(x)$ for each $n>k$, it follows that $\|\rho_{n,k}(x)\|\rho_n(1_{A_n})\geq \rho_n(\rho_{n,k}(x))=\rho_k(x)$ for all $n>k$. Hence $\|\rho_{n,k}(x)\| e \geq \rho_k(x)$ for all $n>k$ and so $\|\rho_k(x)\|e\geq \rho_k(x)$. 

With \cref{prop: manifesting are coi'} it follows that $e-\bar{x}\in \Cstar_{\sbt}(\Lima)_+$ for all self-adjoint $\bar{x}\in \Cstar_{\sbt}(\Lima)^1$, and hence 
$e$ is the unit of $\Cstar_{\sbt}(\Lima)$. 
\end{proof}

\begin{remark}
For a $\Cstar$-encoding system $(A_n,\rho_{n+1,n})_n$, if $e\in \Lima$, then \begin{align*}(\Lima, \{\M_r(\Lima)\cap \M_r(A_\infty)_+\}_r,e)\\=(\Lima, \{\M_r(\Cstar_{\sbt}(\Lima))_+\}_r,e)\end{align*} is an abstract operator system in the sense of Choi and Effros (\cite[Chapter 13]{Pau02}). \end{remark}

\section{$\Cstar$-encoding systems and nuclearity}
In this section, we consider finite-dimensional c.p.c.\ systems, and under this restriction, we will be able to characterize when the limit of a c.p.c.\ system is completely order isomorphic to a (nuclear) $\Cstar$-algebra. We begin with a corollary to Ozawa and Sato's One-Way-CPAP, which appeared implicitly in \cite{Ozawa2002} (via \cite{KS03}); see \cite[Theorem 5.1]{Sato2019} for the explicit statement and its  proof. We recall it here for the reader's convenience. 

\begin{theorem}[\cite{Sato2019, Ozawa2002}]\label{thm: Sato}
A $\Cstar$-algebra $A$ is nuclear if and only if there exists a net $\{\varphi_\lambda\colon F_\lambda\longrightarrow A\}_{\lambda\in \Lambda}$ of c.p.c.\ maps from finite-dimensional $\Cstar$-algebras $\{F_\lambda\}_{\lambda \in \Lambda}$ such that the induced c.p.c.\ map 
\begin{align*}
    \Phi=(\varphi_\lambda)_\lambda\colon \textstyle{\prod} F_\lambda/\textstyle{\bigoplus} F_\lambda &\longrightarrow \ell^\infty (\Lambda,A)/c_0(\Lambda, A),
    \end{align*}
    given by
    $\Phi([(x_\lambda)_{\lambda\in \Lambda}])=[(\varphi_\lambda(x_\lambda))_{\lambda\in \Lambda}],
$ 
satisfies
\[\iota(A^1)\subset \Phi\big((\textstyle{\prod} F_\lambda/ \textstyle{\bigoplus} F_\lambda)^1\big),\]
where $\iota\colon A \longrightarrow \ell^\infty (\Lambda,A)/c_0(\Lambda, A)$ denotes the identification of $A$ with the sub-$\Cstar$-algebra of $\ell^\infty (\Lambda,A)/c_0(\Lambda, A)$ consisting of equivalence classes of constant nets.
\end{theorem}

Using this One-Way-CPAP, we were able to show in \cite[Theorem 2.13]{CW1} that the $\Cstar$-limit of any $\CPCstar$-system is nuclear. In fact, the exact same proof shows a stronger generalization. We give the statement and proof here again to make that clear.

\begin{corollary}\label{thm: nuclear}
If a $\Cstar$-algebra $A$ is completely order isomorphic to the limit $\Lim$ of a finite-dimensional c.p.c.\ system $(F_n,\rho_{n+1,n})_n$, then $A$ is nuclear. 
\end{corollary}

\begin{proof}
Let $(F_n,\rho_{n+1,n})_n$ be a finite-dimensional c.p.c.\ system and $\Psi:A\to F_\infty$ a complete order embedding with $\Psi(A)=\Lim$.  For each $m\geq 0$, we define a c.p.c.\ map $\varphi_m\coloneqq \Psi^{-1}\circ\rho_m\colon F_m\longrightarrow A$.  We denote  the sequence algebra $\prod_m A/ \bigoplus_m A$ by $A_\infty$, and we write $\iota\colon A\longrightarrow A_\infty$ 
for the 
embedding as equivalence classes of constant sequences. 
Let $\Phi\colon F_\infty\longrightarrow A_\infty$ be the c.p.c.\ map induced by the $\varphi_m$ as in \cref{thm: Sato} with 
$\Phi([(x_m)_m])=[(\varphi_m(x_m))_m]$. 
Note that 
for $k\geq 0$ and $x\in F_k$, we have 
\begin{align*}
\Phi(\rho_k(x))&=\Phi([(\rho_{m,k}(x))_{m>k}])\\&=[(\varphi_m(\rho_{m,k}(x)))_{m>k}]\\
&=[((\Psi^{-1}\circ\rho_m)(\rho_{m,k}(x))_{m>k}]\\
&=[(\Psi^{-1}\circ\rho_k(x))_{m>k}]\\
&=\iota\circ\Psi^{-1}\circ\rho_k(x) 
\end{align*}
Since these elements are dense in $\Lim$, it follows that $\Phi|_{\Lim}=\iota\circ\Psi^{-1}$.  
Since $\Psi^{-1}$ is isometric, that gives us  \begin{align*}\iota(A^1) =\iota\circ\Psi^{-1}\big(\Lim ^1\big)= \Phi\big(\Lim ^1\big)\subset \Phi(F_\infty^1).\end{align*} 
Now with \cite[Theorem 5.1]{Sato2019} (as stated above in \cref{thm: Sato}), we conclude that $A$ is nuclear. 
\end{proof}

Combining this with \cref{prop: manifesting are coi'}, we have the following corollary.

\begin{corollary}\label{thm: sato}
If $(F_n,\rho_{n+1,n})_n$ is a finite-dimensional $\Cstar$-encoding 
system, then its $\Cstar$-limit 
is a nuclear $\Cstar$-algebra.
\end{corollary}

Now we proceed towards showing that $\Cstar$-encoding is necessary for the limit of a finite-dimensional c.p.c.\ system to be completely order isomorphic to a (nuclear) $\Cstar$-algebra. The following is essentially \cite[Proposition 5.1.4]{BK97}. 
\begin{proposition}\label{system to approx}
   Let $(F_n,\rho_{n+1,n})_n$ be a finite-dimensional c.p.c.\ system, $A$ a nuclear $\Cstar$-algebra, and $\Psi\colon A\to F_\infty$ a complete order embedding with $\Psi(A)=\Lim$. Define $\varphi_n\coloneqq \Psi^{-1}\circ \rho_n:F_n\to A$. Then  there exist c.p.c.\ maps $\psi_n: A \to F_n$ so that $(A \xrightarrow{\psi_n} F_n \xrightarrow{\varphi_n} A)_n$ forms a system of c.p.c.\ approximations of $A$. If $\Psi$ is order zero, resp.\ a $^*$-homomorphism, then the maps $\psi_n$ are asymptotically order zero, resp.\ multiplicative. 
\end{proposition}

\begin{proof}
    Since $A$ is nuclear by \cref{thm: sato} and separable, the Choi-Effros lifting theorem guarantees a lift $\oplus_m \psi_m\colon A\to \prod F_m$ of the c.p.c.\ map $\Psi\colon A \to F_\infty$. We claim that these along with the $(\varphi_n)_n$ form a system of c.p.c.\ approximations for $A$.  Since $\Psi^{-1}(\bigcup_n \rho_n(F_n))=\bigcup_n \varphi_n(F_n)$ is dense in $A$, it suffices to check the approximations on $\varphi_k(x)$ for some fixed $k\geq 0$ and $x\in F_k$. Let $\varepsilon>0$. Since
    \begin{align*} [(\psi_n(\varphi_k(x)))_n]=\Psi(\varphi_k(x))=\rho_k(x)=[(\rho_{n,k}(x))_{n>k}]
    \end{align*}
   for all $n>k$, there exists $N>k$ such that for all $n>N$,
    \begin{align*}
        \|\rho_{n,k}(x)-\psi_n(\varphi_k(x))\|<\varepsilon. 
    \end{align*} 
    Then since $\Psi$ is isometric, we have for all $n>N$ 
    \begin{align*}
        \|\varphi_k(x)-\varphi_n(\psi_n(\varphi_k(x)))\|&=\|(\Psi^{-1}\circ\rho_k)(x)-(\Psi^{-1}\circ\rho_n)(\psi_n(\varphi_k(x)))\|\\
        &=\|\rho_k(x)-\rho_n(\psi_n(\varphi_k(x)))\|\\
        &< \|\rho_k(x)-\rho_n(\rho_{n,k}(x)))\|+ \varepsilon\\
        &=\varepsilon. 
    \end{align*}
    The final claims are immediate. 
\end{proof}

\begin{remark}
    It follows that only a quasi-diagonal $\Cstar$-algebra can be $^*$-isomorphic to the limit of a finite-dimensional c.p.c.\ system. That means that for $\Cstar$-algebras that are not quasi-diagonal, a complete order isomorphism is the best we can do. 
\end{remark}

\begin{corollary}\label{encoding to approx}
   Let $(F_n,\rho_{n+1,n})_n$ be a finite-dimensional $\Cstar$-encoding system, and define $\varphi_n\coloneqq \Theta^{-1}\circ \rho_n:F_n\to \Cstar_{\sbt}(\Lim)$. Then there exist c.p.c.\ maps $\psi_n: A \to F_n$ so that $(\Cstar_{\sbt}(\Lim) \xrightarrow{\psi_n} F_n \xrightarrow{\varphi_n} \Cstar_{\sbt}(\Lim))_n$ is a system of c.p.c.\ approximations of $\Cstar_{\sbt}(\Lim)$. 
\end{corollary}

\begin{theorem}\label{thm: big theorem}
    Let $(F_n,\rho_{n+1,n})_n$ be a finite-dimensional c.p.c.\ system. Then the following are equivalent. 
    \begin{enumerate}[label=\textnormal{(\roman*)}]
        \item The limit $\Lim$ is completely order isomorphic to a $\Cstar$-algebra.
        \item The limit $\Lim$ is completely order isomorphic to a nuclear $\Cstar$-algebra.
        \item The system has a $\Cstar$-encoding subsystem. 
    \end{enumerate}
\end{theorem}

\begin{proof}
    That (i) $\Leftrightarrow$ (ii) follows from \cref{thm: nuclear}. If the system has a $\Cstar$-encoding subsystem, then we know from \cref{thm: sato} and \cref{prop: manifesting are coi'} that the limit of this subsystem is completely order isomorphic to a nuclear $\Cstar$-algebra. Hence by \cref{subsystem}, the limit of the original system is as well. Thus we already have (iii) $\Rightarrow$ (ii) $\Leftrightarrow$ (i), and it remains to prove (ii) $\Rightarrow$ (iii). 

    Let $A$ be a nuclear $\Cstar$-algebra and $\Psi\colon A\to \Lim$ a complete order isomorphism, which we regard as a complete order embedding $\Psi\colon A\to F_\infty$. For each $m\geq 0$, define $\varphi_m\coloneqq \Psi^{-1}\circ\rho_m\colon F_m\to A$. 
    Just as in the proof of \cref{system to approx}, we have 
    a lift $\oplus_m \psi_m:A\to \prod_m F_m$ of $\Psi\colon A\to F_\infty$ 
    such that
    for any $k\geq 0$, $x\in F_k$, and $\varepsilon>0$, there exists $N>k$ such that for all $n>N$,
    \begin{align*}
        \|\rho_{n,k}(x)-\psi_n(\varphi_k(x))\|<\varepsilon. 
    \end{align*}
Using the compactness of the unit ball of each $(F_k)^1$, for any $k\geq 0$ and  $\varepsilon>0$, we can find $N>k$ so that for all $n>N$,
\begin{align*}
        \|\rho_{n,k}-\psi_n\circ\varphi_k\|<\varepsilon. 
    \end{align*}

Let $(\varepsilon_j)_j\in c_0(\mathbb{N})^1_+$ be a decreasing sequence. Choose $n_0=0$ and $n_1>n_0$ so that $\|\rho_{n,n_0}-\psi_n\circ\varphi_{n_0}\|<\varepsilon_0$ for all $n\geq n_1$. Continue this way to form a subsystem $(F_{n_j},\rho_{n_{j+1},n_j})_j$ such that for any $k\geq 0$, we have for all $j>k$
\begin{align*}
    \|\rho_{n_j,n_k}-\psi_{n_j}\circ\varphi_{n_k}\|<\varepsilon_k.
\end{align*}
We claim that this subsystem $(F_{n_j},\rho_{n_{j+1},n_j})_j$ is $\Cstar$-encoding.

With the same notation as in \cref{subsystem}, we form the limit $\Limj$ of the subsystem via the maps $\hat{\rho}_{n_j}:F_{n_j}\to \prod F_{n_j}/\bigoplus F_{n_j}$.
By \cref{subsystem}, we have a complete order isomorphism $\theta\coloneqq \hat{\pi}|_{\Lim} \colon \Lim \to \Limj$ where $\hat{\pi}\colon \prod F_{m}/\bigoplus F_{m}\to  \prod F_{n_j}/\bigoplus F_{n_j}$ is induced by the surjection $\pi:\prod_m F_m \to \prod_j F_{n_j}$. Then $\theta\circ\Psi\colon A\to \Limj$ is a complete order isomorphism with lift $\oplus_j \psi_{n_j}\colon A\to \prod_j F_{n_j}$. 

After this point, we will consider only the subsystem $(F_{n_j},\rho_{n_{j+1},n_j})_j$, and so we drop the subscripts for ease of notation. With that we have that for any $\varepsilon>0$ there exists $M>0$ so that for all $m>n>M$
\begin{align}\label{*}
    \|\rho_{m,n}-\psi_m\circ\varphi_n\|<\varepsilon.
\end{align} 
(Note that now $m$ does not depend on $n$.) Just as in the proof of \cref{system to approx}, we still have that for any $k\geq 0$, $x\in F_k$, and $\varepsilon>0$, there exists $N>k$ such that for all $n>N$,
    \begin{align}\label{ptwise2}
        \|\rho_{n,k}(x)-\psi_n(\varphi_k(x))\|<\varepsilon, 
    \end{align}
and $(A\xrightarrow{\psi_n}F_n\xrightarrow{\varphi_n}A)_n$ is a system of c.p.c.\ approximations of $A$. 

Now to show that $(F_n,\rho_{n+1,n})_n$ is $\Cstar$-encoding, set $k\geq 0$, $x,y\in F_k^1$, and $0<\varepsilon<1$. By possibly writing $x$ as a linear combination of self-adjoint elements and then distributing, we may without loss of generality reduce to the case where $x$ is self-adjoint. Set $\eta\coloneqq \sfrac{\varepsilon}{6}$. Using our system of c.p.c.\ approximations, \eqref{*}, and \eqref{ptwise2}, we can choose $M>k$ so that for all $m>n> M$, $i\in \{1,2\}$, and $w\in \{x,y\}$ we have 
\begin{align}
    \|\varphi_n(\psi_n\big(\varphi_k(x)\varphi_k(y)\big))-\varphi_k(x)\varphi_k(y)\|&<\sfrac{\eta^2}{3}<\eta \label{1}\\
    \|\varphi_n(\psi_n\big(\varphi_k(x)^i\big))-\varphi_k(x)^i\|&<\sfrac{\eta^2}{3}<\eta\label{2}\\
    \|\rho_{m,n}-\psi_m\circ\varphi_n\|&<\eta, \text{ and} \label{3}\\
    \|\rho_{n,k}(w)-\psi_n(\varphi_k(w))\|&<\eta \label{4}.
\end{align}
With \eqref{2}, we can invoke \cref{lem: Stinespring} to conclude that for any $n>M$ and $v\in F_n^1$ we have 
\begin{align}\label{5}
    &\phantom{\overset{\eqref{2}}{<}}\|\varphi_n\big(\psi_n(\varphi_k(x))v\big)-\varphi_k(x)\varphi_n(v)\|\\
    &\overset{\phantom{\eqref{2}}}{<} \|\varphi_n(\psi_n(\varphi_k(x)))\varphi_n(v)-\varphi_k(x)\varphi_n(v)\| + \eta \nonumber \\
    &\overset{\eqref{2}}{<}2\eta. \nonumber
    \end{align}
Then for $m>n>M$ we have 
\begin{align*}
   &\phantom{\overset{\eqref{4}}{<} } \|\rho_{m,n}\big(\rho_{n,k}(x)\rho_{n,k}(y)\big)-\psi_m\big(\varphi_k(x)\varphi_k(y)\big)\|\\
   &\overset{\eqref{4}}{<} \|\rho_{m,n}\big(\psi_n(\varphi_k(x))\psi_n(\varphi_k(y))\big)-\psi_m\big(\varphi_k(x)\varphi_k(y)\big)\| + 2\eta\\
   &\overset{\eqref{3}}{<} \|(\psi_m\circ\varphi_n)\big(\psi_n(\varphi_k(x))\psi_n(\varphi_k(y))\big)-\psi_m\big(\varphi_k(x)\varphi_k(y)\big)\| + 3\eta\\
   &\overset{\phantom{\eqref{3}}}{\leq} \|\varphi_n\big(\psi_n(\varphi_k(x))\psi_n(\varphi_k(y))\big)-\varphi_k(x)\varphi_k(y)\| + 3\eta\\
    &\overset{\eqref{5}}{<} \|\varphi_k(x)\varphi_n(\psi_n(\varphi_k(y)))-\varphi_k(x)\varphi_k(y)\| + 3\eta +2\eta\\
    &\overset{\eqref{2}}{<} 6\eta\\
    &\overset{\phantom{\eqref{3}}}{=}\varepsilon.
\end{align*}
With a triangle inequality, this shows that the system satisfies \cref{def: encoding}, which shows that (ii) $\Rightarrow$ (iii).
\end{proof}

\begin{remark}\label{big thm for NF}
    Using a similar argument, one can show that the following are equivalent for a c.p.c.\ system $(F_n,\rho_{n+1,n})_n$:
    \begin{enumerate}[label=\textnormal{(\roman*)}]
        \item The limit $\Lim$ is a $\Cstar$-algebra.
        \item The limit $\Lim$ is a nuclear $\Cstar$-algebra.
        \item The system has an NF subsystem. 
    \end{enumerate}
\end{remark}

\begin{remark}
It follows that \emph{only} in the case of NF systems does the product from \eqref{prod'} agree with the product on $\prod_m F_m / \bigoplus_m F_m$ (\cref{big thm for NF}) (meaning the limit is a sub-$\Cstar$-algebra). 
\end{remark}

\begin{corollary}\label{cor: cpcstar is encoding}
  Every $\CPCstar$-system has a $\Cstar$-encoding subsystem. 
\end{corollary}

\begin{proof}
      We know from \cite[Proposition 2.7 and Theorem 2.13]{CW1} that the limit of every $\CPCstar$-system is completely order isomorphic to a (nuclear) $\Cstar$-algebra. Hence \cref{thm: big theorem} tells us that every $\CPCstar$-system has a $\Cstar$-encoding subsystem.
\end{proof}

\begin{remark}
By \cref{cor: cpcstar is encoding} we can say that $\Cstar$-encoding systems generalize $\CPCstar$-systems (in the same sense that $\CPCstar$-systems generalize NF systems in \cite[Theorem 4.4]{CW1}).
\section{A c.p.c.\ system from a system of c.p.c.\ approximations}\label{sect: from cpap}
\end{remark}

Now we arrive at our main class of examples of $\Cstar$-encoding systems. In this section we show that  
\emph{any} summable system of c.p.c.\ approximations of a separable nuclear $\Cstar$-algebra $A$ induces a $\Cstar$-encoding system 
whose $\Cstar$-limit is canonically isomorphic to $A$. We begin by recalling the notion of summable systems of c.p.c.\ approximations from \cite{CW1}. 

\begin{definition}\label{def: summable}\label{def: ass sys}
Let $A$ be a separable $\Cstar$-algebra and  
 $(A\xrightarrow{\psi_n}F_n\xrightarrow{\varphi_n}A)_n$ a system of c.p.c.\ approximations of $A$.  
Set $\rho_{n+1,n}\coloneqq\psi_{n+1}\circ\varphi_n$ for each $n\geq 0$. We call the c.p.c.\ system $(F_n,\rho_{n+1,n})_n$ the \emph{associated c.p.c.\ system} for $(A\xrightarrow{\psi_n}F_n\xrightarrow{\varphi_n}A)_n$. For $m\geq n\geq 0$, we define $\rho_{m,n}$ as in \cref{def: cpc system}.\\
We say the system of c.p.c.\ approximations \emph{summable} if there exists a decreasing sequence $(\varepsilon_m)_m\in \ell^1(\mathbb{N})_+^1 $ such that for all $m>n\geq 0$,
\begin{align*}
  \|\varphi_n-\varphi_m \circ \psi_m\circ\varphi_n\|<\varepsilon_m. \end{align*}
\end{definition}

\begin{remarks}\label{rmk: summable}
(i) 
In essence, summability guarantees that the maps $\rho_{m,n}\coloneqq\psi_m\circ\varphi_{m-1}\circ\hdots \circ\varphi_n\colon F_n\to F_m$ become uniformly close to the maps $\psi_m\circ\varphi_n\colon F_n\to F_m$ for $m>n$ sufficiently large. 
Indeed, since the unit ball of each $F_n$ is compact, 
we may choose for any $\varepsilon>0$ an $M>0$ so that for all $m>n>M$, 
\begin{align}
    \|\rho_{m,n}-\psi_m\circ\varphi_n\|&\leq\|\varphi_{m-1}\circ \rho_{m-1,n}-\varphi_n\|<\textstyle{\sum}_{j=n+1}^{m-1}\varepsilon_j<\varepsilon. \label{eq: zigzag}
\end{align}
A key upshot is that for any $k\geq 0$ and $x\in F_k$, the sequence $(\varphi_n(\rho_{n,k}(x)))_{n>k}$ is Cauchy and hence converges in $A$. We denote the limit as 
\begin{align}\label{ax}
    a_x\coloneqq\lim_n \varphi_n(\rho_{n,k}(x)).
\end{align}
Another upshot is that $\{a_x\mid x\in F_k, k\geq 0\}$ is dense in $A$. Indeed, using \eqref{eq: zigzag} and our completely positive approximations, for any given $a\in A$ and $\varepsilon>0$, we can find $M>0$ such that for all $m>n>M$, 
\begin{align*} 
\|\varphi_m \big( \rho_{m,n} \big( \psi_n(a)\big)\big)-a\|&\leq  \|\big(\varphi_m \circ \rho_{m,n} -\varphi_n\big) (\psi_n(a))\|+\|\varphi_n(\psi_n(a))-a\|\\
&< \sfrac{\varepsilon}{2} +\sfrac{\varepsilon}{2} \\&<\varepsilon. 
    \end{align*}
and so $a=\lim_n a_{\psi_n(a)}$.

(ii) Given any system of c.p.c.\ approximations $(A\xrightarrow{\psi_n}F_n\xrightarrow{\varphi_n}A)_n$ of a separable nuclear $\Cstar$-algebra $A$ and a decreasing sequence $(\varepsilon_m)_m\in \ell^1(\mathbb{N})_+^1$, using the compactness of each $F_n^1$, one can always find a $(\varepsilon_m)_m$-summable subsystem of $(A\xrightarrow{\psi_n}F_n\xrightarrow{\varphi_n}A)_n$ (as noted in \cite[Remark 3.2(ii)]{CW1}). 
\end{remarks}

The aim of this section is to show that when a system of c.p.c.\ approximations is summable, the associated c.p.c.\ system is $\Cstar$-encoding. Its $\Cstar$-limit is $^*$-isomorphic to $A$ with the $^*$-isomorphism given by composing the map $\Psi\colon A\to F_\infty$ induced by the sequence $(\psi_n)_n$ with the map $\Theta^{-1}$ from \cref{prop: manifesting are coi'}. 
First, we show  that $(F_n,\rho_{n+1,n})_n$ is $\Cstar$-encoding. 
This will follow from the following estimates, which we state here for future use. The inequality in \eqref{eq: mult eq'} follows from the definition of $a_x$ and $a_y$, but we label it here for easy reference later.

\begin{lemma}\label{lem: rho and psi prod}
Let $(A\xrightarrow{\psi_n}F_n\xrightarrow{\varphi_n}A)_n$ be a summable system of c.p.c.\ approximations of a $\Cstar$-algebra $A$ with associated c.p.c.\ system $(F_n,\rho_{n+1,n})_n$. Then for any $k\geq 0$, $x,y\in F_k$, and $\varepsilon>0$, there exists $M>k$ so that for all $m>n>M$, 
\begin{align}
    \|\rho_{m,n}\big(\rho_{n,k}(x)\rho_{n,k}(y)\big) - \psi_m(a_xa_y)\|&<\varepsilon,\ \text{ and }\label{eq: mult eq}\\ 
    \|\rho_{m,k}(x)\rho_{m,k}(y)-\psi_m(a_x)\psi_m(a_y)\|&<\varepsilon, \label{eq: mult eq'} 
\end{align}
where 
$a_x,a_y\in A$ are as defined in 
\eqref{ax}.
\end{lemma}

\begin{proof}
Set $k\geq 0$, $x,y\in F_k^1$, and $0<\varepsilon<1$.  Moreover, by possibly writing $x$ as a linear combination of self-adjoint elements and distributing, we may reduce the argument to the case where $x$, and hence $a_x$, is self-adjoint. 
Using \eqref{eq: zigzag} and the c.p.c.\ approximations, we may choose $M_0>0$ so that for all $m>n>M_0$,
\begin{align*}
    \|\rho_{m,n}-\psi_m\circ\varphi_n\|&<\sfrac{\varepsilon}{4},\\
    \|\varphi_n(\psi_n(a_x))\varphi_n(\psi_n(a_y))-a_xa_y\|&<\sfrac{\varepsilon}{4},\ \text{ and }\\
    \|\varphi_n(\psi_n(a_x^i))-a_x^i\|&<(\sfrac{\varepsilon}{4})^2/3<\sfrac{\varepsilon}{8},\ \text{ for } i=1,2.
\end{align*}
By definition of $a_x$ and $a_y$, we can find $M>M_0$ so that for all $m>M_0$,
\begin{align*}
    &\|\rho_{m,k}(x)\rho_{m,k}(y)-\psi_m(a_x)\psi_m(a_y)\|\\
    &=\|\psi_m(\varphi_{m-1}(\rho_{m-1,k}(x)))\psi_m(\varphi_{m-1}(\rho_{m-1,k}(y)))-\psi_m(a_x)\psi_m(a_y)\|\\
    &<\sfrac{\varepsilon}{4}.
\end{align*}
This establishes \eqref{eq: mult eq'}. Using \cref{lem: Stinespring} and the preceding estimates we have for all $m>n>M$ 
\begin{align*}
    \|\rho_{m,n}(\rho_{n,k}(x)\rho_{n,k}(y)) - \psi_m(a_xa_y)\|
    &< \|\rho_{m,n}(\psi_n(a_x)\psi_n(a_y))-\psi_m(a_xa_y)\| +\sfrac{\varepsilon}{4} \\
    &<  \|\psi_m\circ \varphi_n(\psi_n(a_x)\psi_n(a_y))-\psi_m(a_xa_y)\| + \sfrac{\varepsilon}{2}\\
    &\leq \|\varphi_n(\psi_n(a_x)\psi_n(a_y))-a_xa_y\| + \sfrac{\varepsilon}{2}\\
    &< \|\varphi_n(\psi_n(a_x))\varphi_n(\psi_n(a_y))-a_xa_y\| +\sfrac{3\varepsilon}{4} \\
    &<\varepsilon. \qedhere
\end{align*}
\end{proof}
Using a triangle inequality on \eqref{eq: mult eq}, it follows that any summable system of c.p.c.\ approximations of a nuclear $\Cstar$-algebra is $\Cstar$-encoding  
as in \cref{def: encoding}. 
\begin{corollary}\label{prop: times + * are necessary}
Let $(A\xrightarrow{\psi_n}F_n\xrightarrow{\varphi_n}A)_n$ be a summable system of c.p.c.\ approximations of a $\Cstar$-algebra $A$. Then the associated c.p.c.\ system $(F_n,\rho_{n+1,n})_n$ is $\Cstar$-encoding.
\end{corollary}

To define the isomorphism between $A$ and $\Cstar_{\sbt}\big(\Lim\big)$, we start with the map $\Psi\colon A\to F_\infty$ induced by the maps $\psi_n\colon A\to F_n$, i.e., 
  \begin{align}
 \Psi(a)\coloneqq[(\psi_n(a))_n], \ \text{ for all } a\in A. \label{eq: Psi}
  \end{align}

Since the system is summable, we know from \cite[Lemma 3.4]{CW1} that $\Psi$ is a complete order embedding, and its image is exactly the limit $\Lim\subset F_\infty$.  Now with \cref{cor: limit is Cstar} and Remark \ref{rmk: surj coi}(ii), we conclude that $\Theta^{-1}\circ\Psi\colon A\to \Cstar_{\sbt}(\Lim)$ is a $^*$-isomorphism where $\Theta$ is the identity map in \cref{prop: manifesting are coi'}. Combined with \cref{prop: times + * are necessary}, this gives the main theorem of this section.

\begin{theorem}\label{thm: main}
Let $(A\xrightarrow{\psi_n}F_n\xrightarrow{\varphi_n}A)_n$ be a summable system of c.p.c.\ approximations of a $\Cstar$-algebra $A$. 
Then the induced c.p.c.\ system $(F_n,\rho_{n+1,n})_n$ is $\Cstar$-encoding, 
and moreover the map $\Theta^{-1}\circ\Psi\colon A\to \Cstar_{\sbt}(\Lim)$ is a $^*$-isomorphism between $A$ and the $\Cstar$-limit $\Cstar_{\sbt}(\Lim)$ of the system, where $\Psi\colon A\to F_\infty$ is the map from \eqref{eq: Psi} and $\Theta$ is the identity map in \cref{prop: manifesting are coi'}. 
\end{theorem}

Combining this with  \cref{subsystem}, we have the following. 

\begin{corollary}
    Let $(A\xrightarrow{\psi_n}F_n\xrightarrow{\varphi_n}A)_n$ be any system c.p.c.\ approximations of a nuclear $\Cstar$-algebra $A$, and let $(F_n,\psi_{n+1}\circ\varphi_n)_n$ be the associated subsystem.  Then $A$ is completely order isomorphic to the limit of the c.p.c.\ system.
\end{corollary}



\section{$\Cstar$-encoding systems for amenable group $\Cstar$-algebras}\label{sect: Folner Z}

One advantage of our construction is how readily we can construct examples of $\Cstar$-encoding systems. In this section we show how to build a $\Cstar$-encoding system for the reduced group $\Cstar$-algebra for any countable discrete amenable group using classic c.p.c.\ approximations via F{\o}lner sequences (as in \cite[Theorem 2.6.8]{BO}). 
We shall also see that, at least for any non-torsion group, the resulting c.p.c.\ system is never NF or $\CPCstar$. 

Let $G$ be a countable amenable discrete group with reduced group $\Cstar$-algebra $\Cstar_r(G)\subset B(\ell^2(G))$ induced by the left regular representation $\lambda:G\to B(\ell^2(G))$ with $\lambda(g)=\lambda_g\in \mathcal{U}(\ell^2(G))$ given by $\lambda_g\delta_h=\delta_{gh}$ for all $g,h\in G$. Recall that a \emph{F{\o}lner sequence} for $G$ is a sequence $(\mathcal{F}_n)_n$ of finite subsets of $G$ so that for any $s\in G$ 
 \begin{align*}
   \frac{|\mathcal{F}_n\Delta s\mathcal{F}_n|}{|\mathcal{F}_n|} =2-2\frac{|\mathcal{F}_n\cap s\mathcal{F}_n|}{|\mathcal{F}_n|}\xrightarrow[n\to\infty]{} 0, 
 \end{align*}
We will call a F{\o}lner sequence \emph{summable} if there exists a decreasing sequence $(\varepsilon_n)\in \ell^1(\mathbb{N})^1_+$ so that for all $m>n\geq 0$
 \begin{align}
      \max_{g,h\in \mathcal{F}_{n}}
     \left(1-\frac{|\mathcal{F}_{m}\cap gh^{-1}\mathcal{F}_{m}|}{|\mathcal{F}_{m}|}\right)|\mathcal{F}_{n}|
     <\varepsilon_m.\label{eq: sum Fol}
 \end{align}
Given any F{\o}lner sequence we can find a summable subsequence.

Now, from a summable F{\o}lner sequence $(\mathcal{F}_n)_n$ for $G$ with respect to a decreasing sequence $(\varepsilon_n)_n\in \ell^1(\mathbb{N})_+^1$, we construct a summable system of c.p.c.\ approximations of $\Cstar_r(G)$ following \cite[Theorem 2.6.8]{BO}: For each $n\geq 1$, let $P_n\colon \ell^2(G)\to \ell^2(G)$ be the projection onto the span of $\{\delta_g\mid g\in \mathcal{F}_n\}$ and identify $P_nB(\ell^2(G))P_n$ with $M_{\mathcal{F}_n}(\C)$ with canonical matrix units $\{e_{g,h}\}_{g,h\in \mathcal{F}_n}$. Define $\psi_n\colon \Cstar_r(G)\to M_{\mathcal{F}_n}(\C)$ on $\lambda(G)\subset \Cstar_r(G)$ by 
\begin{align}
    \psi_n(\lambda_s)=P_n\lambda_s P_n=\sum_{r\in \mathcal{F}_n\cap s\mathcal{F}_n} e_{r,s^{-1}r}, \label{eq: psi Fol}
    \end{align}
   for $s\in G$, and $\varphi_n:M_{\mathcal{F}_n}(\C)\to \Cstar_r(G)$ on matrix units by
    \begin{align*}
        \varphi_n(e_{g,h})=\frac{1}{|\mathcal{F}_n|}\lambda_{gh^{-1}}
\end{align*}
for $g,h\in \mathcal{F}_n$. 
These maps are u.c.p., and the F{\o}lner condition guarantees that this is a system of c.p.c.\ approximations of $\Cstar_r(G)$ (see \cite[Theorem 2.6.8]{BO} for a proof). To see that the system is summable, we note that for each $s\in G$ and $m\geq 0$, we have  \begin{align*}
     \varphi_m(\psi_m(\lambda_s))=\frac{|\mathcal{F}_m\cap s\mathcal{F}_m|}{|\mathcal{F}_m|}\lambda_s,
 \end{align*}
 and we approximate for $m>n\geq 0$
\begin{align*}
    \|\varphi_n-\varphi_m\circ\psi_m\circ\varphi_n\|&\overset{\phantom{\eqref{eq: sum Fol}}}{\leq} |\mathcal{F}_n|^2 \max_{g,h\in \mathcal{F}_n} \|\varphi_n(e_{g,h})-\varphi_m\circ\psi_m\circ\varphi_n(e_{g,h})\|\\
    &\overset{\phantom{\eqref{eq: sum Fol}}}{=}|\mathcal{F}_n|^2 \max_{g,h\in \mathcal{F}_n} \left\|\left(1-\frac{|\mathcal{F}_m\cap gh^{-1}\mathcal{F}_m|}{|\mathcal{F}_m|}\right)\frac{\lambda_{gh^{-1}}}{|\mathcal{F}_n|}\right\|\\
    &\overset{\eqref{eq: sum Fol}}{<}\varepsilon_m.
\end{align*}
It follows that the system of approximations is summable and so by \cref{thm: main} the associated system $(\M_{\mathcal{F}_n},\psi_{n+1}\circ\varphi_n)_n$ is $\Cstar$-encoding, and the $\Cstar$-limit is isomorphic to $\Cstar_r(G)$. 

We indicated in the introduction that such systems are generally not NF or $\CPCstar$, and it remains to justify this assertion. Since our maps are all u.c.p., these two notions coincide, and so we focus our attention on NF systems and begin with the following proposition.

\begin{proposition}\label{prop: mult iff mult}
Suppose $A\xrightarrow{\psi_n} F_n \xrightarrow{\varphi_n} A$ is a system of summable c.p.c.\ approximations of a separable nuclear $\Cstar$-algebra. The associated system is asymptotically multiplicative if and only if the sequence $(\psi_n)_n$ is approximately multiplicative in the sense that $\lim_n\|\psi_n(ab)-\psi_n(a)\psi_n(b)\|=0$ for all $a,b\in A$. 
\end{proposition}

\begin{proof}
First we assume the sequence $(\psi_n)_n$ is approximately multiplicative. Let $k\geq 0$, $x,y\in F_k$, and $\varepsilon>0$, and let $a_x,a_y\in A$ as in \eqref{ax}. \cref{lem: rho and psi prod} guarantees an $M>k$ so that \eqref{eq: mult eq} and \eqref{eq: mult eq'} hold for $\sfrac{\varepsilon}{4}$ for all $m>n>M$. By approximate multiplicativity, we can choose $N>M$ so that $\|\psi_m(a_xa_y)-\psi_m(a_x)\psi_m(a_y)\|<\sfrac{\varepsilon}{2}$ for all $m>N$. Then for all $m>n>N$, we have 
\begin{align*}
 \|\rho_{m,n}(\rho_{n,k}(x)\rho_{n,k}(y)) -\rho_{m,k}(x)\rho_{m,k}(y)\|&< \sfrac{\varepsilon}{2} + \|\psi_m(a_xa_y)-\psi_m(a_x)\psi_m(a_y)\|<\varepsilon.
\end{align*}
Hence the associated system $(F_n,\rho_{n+1,n})_n$ is asymptotically multiplicative. 

Now assume $(\psi_n)_n$ is not approximately multiplicative. 
By Remark \ref{rmk: summable}(i) there exist $k\geq 0$, $x,y\in F_k$ that witness this, i.e., there exist $k\geq 0$, $x,y\in F_k$, and $\varepsilon>0$ so that for any $n\in\mathbb{N}$ there exists an $m>n$ with $\|\psi_m(a_xa_y)-\psi_m(a_x)\psi_m(a_y)\|>\varepsilon$. 
\cref{lem: rho and psi prod} guarantees an $M>k$ so that \eqref{eq: mult eq} and \eqref{eq: mult eq'} hold for $\sfrac{\varepsilon}{4}$ for all $m>n>M$, which gives
\begin{align*}
    \|\rho_{m,n}(\rho_{n,k}(x)\rho_{n,k}(y)) - \rho_{m,k}(x)\rho_{m,k}(y)\|< \|\psi_m(a_xa_y)-\psi_m(a_x)\psi_m(a_y)\| + \sfrac{\varepsilon}{2}.
\end{align*}
 for all $m>n>M$. Now for any $n>M$, there exists an $m>n$ such that $\|\psi_m(a_xa_y)-\psi_m(a_x)\psi_m(a_y)\|>\varepsilon$, and so
\[\|\rho_{m,n}(\rho_{n,k}(x)\rho_{n,k}(y)) - \rho_{m,k}(x)\rho_{m,k}(y)\|>\sfrac{\varepsilon}{2}.\]
It follows that the system $(F_n,\rho_{n+1,n})_n$ is not asymptotically multiplicative. 
\end{proof}

Of course some amenable groups, such as finite groups, will admit a summable system of c.p.c.\ approximations from a F{\o}lner sequence as above with $(\psi_n)_n$ approximately multiplicative, but it turns out this will never hold for many groups. 
\begin{proposition}
Let $G$ be a countable discrete amenable group with summable F{\o}lner sequence $(\mathcal{F}_n)_n$, and let $\big(\Cstar_r(G)\xrightarrow{\psi_{n}}\M_{\mathcal{F}_{n}}\xrightarrow{\varphi_{n}}\Cstar_r(G)\big)_n$ be the summable approximation for $\Cstar_r(G)$ derived above. If $G$ contains an element with infinite order, then the maps $(\psi_{n})_n$ are not approximately multiplicative, and the associated $\Cstar$-encoding system is neither NF nor $\CPCstar$.
\end{proposition}

\begin{proof}
Let $s\in G$ with infinite order. It follows in particular that $|\mathcal{F}_n\cap s\mathcal{F}_n|<|\mathcal{F}_n|$ for all $n\geq 0$. 
Then using \eqref{eq: psi Fol} we compute for any $n\geq 0$
\begin{align*}
    \|\psi_n(\lambda_s^*\lambda_s)-\psi_n(\lambda_s^*)\psi_n(\lambda_s)\|
    &=\left\|1_{\M_{\mathcal{F}_n}}-\left(\sum_{r\in \mathcal{F}_n\cap s\mathcal{F}_n} e_{s^{-1}r,r}\right)\left(\sum_{r\in \mathcal{F}_n\cap s\mathcal{F}_n} e_{r,s^{-1}r}\right)\right\|\\
    &=\left\|1_{\M_{\mathcal{F}_n}}-\sum_{r\in \mathcal{F}_n\cap s\mathcal{F}_n} e_{s^{-1}r,s^{-1}r}\right\|
    \\&=1.
\end{align*}
It follows from \cref{prop: mult iff mult} that the associated $\Cstar$-encoding system \\ $(\M_{\mathcal{F}_n},\psi_{n+1}\circ\varphi_n)_n$ is not NF, and since the maps $\rho_{n+1,n}=\psi_{n+1}\circ\varphi_n$ are all u.c.p., the system is also not $\CPCstar$. 
\end{proof}


\bibliographystyle{plain}
\makeatletter\renewcommand\@biblabel[1]{[#1]}\makeatother
\bibliography{References2}

\end{document}